\documentclass[10pt]{amsart}

\usepackage[colorlinks]{hyperref}
\usepackage{color,graphicx,shortvrb}
\usepackage[latin 1]{inputenc}

\usepackage[active]{srcltx} 

\usepackage{enumerate}
\usepackage{amssymb}
\usepackage{tikz-cd}

\newtheorem{theorem}{Theorem}[section]

\newtheorem{corollary}[theorem]{Corollary}

\newtheorem{lemma}[theorem]{Lemma}

\newtheorem{proposition}[theorem]{Proposition}
\newtheorem{remark}[theorem]{Remark}

\def\J#1#2#3{ \left\{ #1,#2,#3 \right\} }

\def\NN{{\mathbb{N}}}
\def\11{\textbf{$1$}}
\def\11b#1{\mathbf{1}_{_{#1}}}

\DeclareGraphicsExtensions{.jpg,.pdf,.png,.eps}

\begin{document}

\title[Surjective isometries between sets of invertible elements]{Surjective isometries between sets of invertible elements in unital Jordan-Banach algebras}

\author[A.M. Peralta]{Antonio M. Peralta}

\address[A.M. Peralta]{Departamento de An{\'a}lisis Matem{\'a}tico, Facultad de
Ciencias, Universidad de Granada, 18071 Granada, Spain.}
\email{aperalta@ugr.es}




\date{November 16th, 2020}

\begin{abstract} Let $M$ and $N$ be unital Jordan-Banach algebras, and let $M^{-1}$ and $N^{-1}$ denote the sets of invertible elements in $M$ and $N$, respectively. Suppose that $\mathfrak{M}\subseteq M^{-1}$ and $\mathfrak{N}\subseteq N^{-1}$ are clopen subsets of $M^{-1}$ and $N^{-1}$, respectively, which are closed for powers, inverses and products of the form $U_{a} (b)$. In this paper we prove that for each surjective isometry $\Delta : \mathfrak{M}\to \mathfrak{N}$ there exists a surjective real-linear isometry $T_0: M\to N$ and an element $u_0$ in the McCrimmon radical of $N$ such that $\Delta (a) = T_0(a) +u_0$ for all $a\in \mathfrak{M}$.\smallskip

Assuming that $M$ and $N$ are unital JB$^*$-algebras we establish that for each surjective isometry $\Delta : \mathfrak{M}\to \mathfrak{N}$ the element $\Delta(\textbf{1}) =u$ is a unitary element in $N$ and there exist a central projection $p\in M$ and a complex-linear Jordan $^*$-isomorphism $J$ from $M$ onto the $u^*$-homotope $N_{u^*}$ such that $$\Delta (a) = J(p\circ a) + J ((\textbf{1}-p) \circ a^*),$$ for all $a\in \mathfrak{M}$. Under the additional hypothesis that there is a unitary element $\omega_0$ in $N$ satisfying $U_{\omega_0} (\Delta(\textbf{1})) = \textbf{1}$, we show the existence of a central projection $p\in M$ and a complex-linear Jordan $^*$-isomorphism $\Phi$ from $M$ onto $N$ such that $$\Delta (a) = U_{w_0^{*}} \left(\Phi (p\circ a) + \Phi ((\textbf{1}-p) \circ a^*)\right),$$ for all $a\in \mathfrak{M}$.
\end{abstract}

\maketitle
\thispagestyle{empty}

\section{Introduction}

The topological properties of the invertible subgroup, $A^{-1}$, of an associative unital Banach algebra $A$ are not enough to distinguish $A$ from another associative unital Banach algebra. Actually, there are examples of associative unital complex Banach algebras $A$ and $B$ which are not isomorphic as real algebras while the invertible groups $A^{-1}$ and $B^{-1}$ are homeomorphically isomorphic as topological groups (see \cite[Remark I.7.8]{Zelaz73}). After the aforementioned counterexample it was clear that, in order to identify two associative unital Banach algebras by a bijective mapping between their invertible groups, we must assume additional properties on this bijection. A very natural hypothesis is to assume that this bijection preserves the distances induced by the norms of the Banach algebras. This was first confirmed by O. Hatori who proved in \cite{Hat09} that, for each surjective isometry $\Delta$ from an open subgroup of the group of invertible elements in an associative unital semisimple commutative Banach algebra $A$ onto an open subgroup of the group of invertible elements in an associative unital Banach algebra $B$, the mapping $\Delta(\textbf{1})^{-1} \Delta$ is an isometric group isomorphism. Consequently, $\Delta(\textbf{1})^{-1} \Delta$ extends to an isometric real-linear algebra isomorphism from $A$ onto $B$.\smallskip

In the non-commutative setting, let $A$ and $B$ be associative unital Banach algebras and suppose that $\mathfrak{A}$ and $\mathfrak{B}$ are open multiplicative subgroups of $A^{-1}$ and $B^{-1}$, respectively. Another remarkable result, also due to O. Hatori, asserts that for each surjective isometry $\Delta: \mathfrak{A}\to \mathfrak{B}$ there exist a surjective real-linear isometry $T_0$ from $A$ onto $B$ and an element $u_0$ in the Jacobson radical of $B$ for which the identity $$\Delta (a) = T_0 (a)+ u_0$$ holds for every $a \in \mathfrak{A}$ (see \cite[Theorem 3.2]{Hat2011}). O. Hatori and K. Watanabe concretized this description in the most favorable case of surjective isometries between open multiplicative subgroups of the invertible elements of two unital C$^*$-algebras (see \cite{HatWan2012}). The result reads as follows:

\begin{theorem}\label{t HatWat}{\rm(\cite[Theorem 2.2]{HatWan2012})} Let $A$ and $B$ be unital C$^*$-algebras, and let $\mathfrak{A}$ and $\mathfrak{B}$ be open subgroups of $A^{-1}$ and $B^{-1},$ respectively. Suppose $\Delta$ is a bijection from $\mathfrak{A}$ onto $\mathfrak{B}$.
Then $\Delta$ is an isometry if and only if $\Delta(\textbf{1})$ is unitary in $B$ and there are a central projection $p$ in $B$, and a complex-linear Jordan $^*$-isomorphism $J$ from $A$ onto $B$ such that $$ \Delta (a) = \Delta(\textbf{1}) p J(a) + \Delta(\textbf{1}) (\textbf{1}-p) J(a)^*, \hbox{ for all } a \in  \mathfrak{A}.$$ Furthermore the operator $ \Delta(\textbf{1}) p J(\cdot) + \Delta(\textbf{1}) (\textbf{1}-p) J(\cdot)^*$ defines a surjective real-linear isometry from $A$ onto $B$.
\end{theorem}

O. Hatori and L. Moln{\'a}r considered subtle variants of the above problem on non-linear preservers (see \cite{HatMol2012,HatMol2014}). These authors achieved fascinating conclusions by showing, for example, that for each surjective isometry $\Delta$ between the unitary groups $\mathcal{U} (A)$ and $\mathcal{U} (B)$ of two unital C$^*$-algebras $A$ and $B$, then $\Delta$ maps the set $e^{i A_{sa}}$ onto the set $e^{i B_{sa}}$, and there exists a central projection $p \in B$ and a Jordan $^*$-isomorphism $J : A\to B$ satisfying $$\Delta (e^{ix}) = \Delta(1) (p J(e^{ix}) + (1-p) J(e^{ix})^*), \ \ (x\in A_{sa})$$ (cf. \cite[Theorem 1]{HatMol2014}). As a consequence, they also proved that every surjective isometry between the unitary groups of two von Neumann algebras admits an extension to a surjective real-linear isometry between these algebras (see \cite[Corollary 3]{HatMol2014}).\smallskip

There is no direct connection between the results on surjective isometries between open subgroups of invertible elements in unital C$^*$-algebras $A$ and $B$ in \cite{Hat2011,HatWan2012} and those about surjective isometries between the unitary groups of $A$ and $B$ in \cite{HatMol2014}, none of them seems to be deducible from the other. Despite in the case of von Neumann algebras both results lead to the same conclusion.\smallskip

The reader has probably realized the strong connections of the above commented result with the Jordan structure underlying associative Banach algebras and C$^*$-algebras. It was I. Kaplansky who first suggested a definition of  a  suitable  Jordan  analogue  of  C$^*$-algebras, which  later materialized in the notion of JB$^*$-algebra (see sections \ref{section McCrimmon radical} and \ref{sec: JBstar algebras} for the concrete notions and basic references). At this stage we shall simply recall that a Jordan-Banach algebra is a (non-necessarily associative) algebra $M$ whose product is abelian and satisfies the so-called \emph{Jordan identity}: $(a \circ b)\circ a^2 = a\circ (b \circ a^2)$ ($a,b\in M$), and is equipped with a complete norm, $\|.\|$, satisfying $\|a\circ b\| \leq \|a\| \ \|b\|$ ($a,b\in M$). If $M$ is unital with unit $\textbf{1}$ we also assume $\|\textbf{1}\|=1$. Every associative Banach algebra $A$ is a Jordan-Banach algebra with respect to the natural Jordan product given by $a\circ b = \frac12 (a b +ba )$. It is worth to note that, even in the case of associative Banach algebras, invertible elements are not stable under Jordan products, what constitutes a real handicap when working with Jordan structures. However, an expression of the form $U_a (b) = 2 (a\circ b) \circ a - a^2 \circ b$ defines and invertible element in $M$ if and only if $a$ and $b$ are invertible (see \cite[Theorem 4.1.3]{Cabrera-Rodriguez-vol1}). A \emph{JB$^*$-algebra} is a complex Jordan-Banach algebra $M$ equipped with an algebra involution $^*$ satisfying  $\|U_{a} (a^*) \|= \|a\|^3$, $a\in M$.\smallskip

The category of C$^*$-algebras is a strict subclass of that given by all JB$^*$-algebras. In a recent collaboration with M. Cueto-Avellaneda we generalized the mentioned Hatori-Moln{\'a}r theorem to the setting of unital JB$^*$-algebras, and we prove, among other things, that every surjective isometry between the sets of unitaries of two unital JB$^*$-algebras admits a (unique) extension to a surjective real-linear isometry between them (see \cite[Theorem 3.9]{CuPe20a}). In the category of Jordan algebras, we lack of a good description of those surjective isometries between certain subsets of invertible elements in two unital Jordan-Banach algebras. This paper is aimed to throw some new light to those questions which remains unsolved in the setting of Jordan-Banach algebras.\smallskip

Let $M$ and $N$ be unital Jordan-Banach algebras, and let $M^{-1}$ and $N^{-1}$ denote the sets of invertible elements in $M$ and $N$, respectively. Suppose that $\mathfrak{M}\subseteq M^{-1}$ and $\mathfrak{N}\subseteq N^{-1}$ are clopen subsets of $M^{-1}$ and $N^{-1}$, respectively, which are closed for powers, inverses and expressions of the form $U_a(b)$. In Theorem \ref{t surjective isometries between invertible clopen} we establish that for each surjective isometry $\Delta : \mathfrak{M}\to \mathfrak{N}$, there exist a surjective real-linear isometry $T_0: M\to N$ and an element $u_0$ in the McCrimmon radical of $N$ such that $\Delta (a) = T_0(a) +u_0$ for all $a\in \mathfrak{M}$. That is, every surjective isometry between certain subsets of invertible elements in unital Jordan-Banach algebras can be extended to a surjective real-linear isometry between these algebras up to a translation by an element in the McCrimmon radical.\smallskip

JB$^*$-algebras constitute a subclass of complex Jordan-Banach algebras, which is the best known and studied. In Jordan theory, JB$^*$-algebras play the role of C$^*$-algebras in the category of associative Banach algebras. We shall devote our final section \ref{sec: JBstar algebras} to prove a more concrete description of those surjective isometries between certain subsets of invertible elements in two unital JB$^*$-algebras $M$ and $N$. If we assume that $\mathfrak{M}\subseteq M^{-1}$ and $\mathfrak{N}\subseteq N^{-1}$ are clopen subsets of $M^{-1}$ and $N^{-1}$, respectively, which are closed for powers, inverses and expressions of the form $U_a(b)$, and $\Delta : \mathfrak{M}\to \mathfrak{N}$ is a surjective isometry, we show that $\Delta(\textbf{1}) =u$ is a unitary element in $N$ and there exist a central projection $p\in M$ and a complex-linear Jordan $^*$-isomorphism $J$ from $M$ onto the $u^*$-homotope $N_{u^*}$ such that $$\Delta (a) = J(p\circ a) + J ((\textbf{1}-p) \circ a^*),$$ for all $a\in \mathfrak{M}$. If we additionally suppose that there exists a unitary $\omega_0$ in $N$ such that the identity $U_{\omega_0} (\Delta(\textbf{1})) = \textbf{1}$ holds, then there exist a central projection $p\in M$ and a complex-linear Jordan $^*$-isomorphism $\Phi$ from $M$ onto $N$ such that $$\Delta (a) = U_{w_0^{*}} \left(\Phi (p\circ a) + \Phi ((\textbf{1}-p) \circ a^*)\right),$$ for all $a\in \mathfrak{M}$ (see Theorem \ref{t surjective isometries between invertible clopen JBstar}).\smallskip

It should be observed here that not only the results possess their own interest and independence from the previous contributions in the associative setting, but we also present a completely new strategy based on a better understanding of these kind of surjective isometries on particular convex subsets to apply a celebrated result by P. Mankiewicz in \cite{Mank1972} (cf. Lemma \ref{l affine on segements}). That is, the proofs here also provide a new independent, and perhaps shorter, strategy to rediscover Hatori's theorem in \cite{Hat2011}, where the arguments are based on a technical result inspired by the new proof of the Mazur--Ulam theorem given by V\"{a}is\"{a}l\"{a} in \cite{Vaisala2003}.\smallskip

In order to conclude the description of this note's structure, we add that section \ref{section McCrimmon radical} is devoted to present the McCrimmon radical of a Jordan-Banach algebra together with its main properties and characterizations.

\section{McCrimmon radical in Jordan-Banach algebras}\label{section McCrimmon radical}

During the thirties of the twentieth century, P. Jordan, J. von Neumann, E. Wigner and some other authors introduced the notion of Jordan algebra as a mathematical model for quantum mechanics. A complex (respectively,  real) \emph{Jordan algebra} $M$ is a (non-necessarily associative) algebra over the complex (respectively, real) field whose product is abelian and satisfies the so-called \emph{Jordan identity}: $(a \circ b)\circ a^2 = a\circ (b \circ a^2)$ ($a,b\in M$). A \emph{normed Jordan algebra} is a Jordan algebra $M$ equipped with a norm, $\|.\|$, satisfying $\|a\circ b\| \leq \|a\| \ \|b\|$ ($a,b\in M$). A \emph{Jordan-Banach algebra} is a normed Jordan algebra $M$ whose norm is complete. If $M$ is unital with unit $\textbf{1}$, we also require $\|\textbf{1}\|=1$. In the case that $M$ does not possess a unit, it can be always embedded in a unital Jordan-Banach algebra. Every real or complex associative Banach algebra is a real or complex Jordan-Banach algebra with respect to the Jordan product $a\circ b: = \frac12 (a b +ba)$; Jordan-Banach algebras obtained in this way are called \emph{special}. We shall always asume that our Jordan-Banach algebra is unital. A Jordan homomorphism between Jordan Banach algebras is a linear map preserving Jordan products (equivalently, squares of elements). A real-linear Jordan homomorphism is a real-linear mapping preserving Jordan products. We shall follow classical notation like in \cite{McCrimm66,McCrimm69,McCrimm71,HogbMcCrimm81,Aupetit93}--the reader should be warned that our notation might slightly differ from the one employed in the recent monograph \cite{Cabrera-Rodriguez-vol1}, where for example, the McCrimmon radical is called the Jacobson radical.\smallskip

Let $M$ be a Jordan Banach algebra. Two key notions to work with a Jordan algebra $M$ are the (Jordan) multiplication operator by an element $a\in M$ (denoted by $L_a$) defined by $L_a (b)= a \circ b$ ($b\in M$), and the $U_a$ operator given by $$U_{a} (x) = (2 L_a^2 -L_{a^2}) (x) = 2 (a\circ x) \circ a  - a^2\circ x \ \  (x\in M).$$ Furthermore, given $a,b\in M$, we shall write $U_{a,b}:M\to M$ for the (complex-)linear mapping on $M$ defined by $$U_{a,b} (x)=(a\circ x) \circ b + (b\circ x)\circ a - (a\circ b)\circ x,\ \ (x\in M).$$ One of the main identities in Jordan algebras assures that \begin{equation}\label{eq fundamental identity UaUbUa} \ \  U_a U_b U_a = U_{U_a(b)}\ \ \ \hbox{("fundamental formula")}, \end{equation} for all $a,b$ in $M$ (cf. \cite[2.4.18]{HOS} or \cite{McCrimm66} or \cite[Proposition 3.4.15]{Cabrera-Rodriguez-vol1}).\smallskip

Henceforth, the powers of an element $a$ in a Jordan algebra $M$ will be denoted as follows:
$$ a^1 =a;\ \  a^{n+1} =a\circ a^n, \quad n\geq 1.$$
If $M$ is unital, we set $a^0=\textbf{1}$. An algebra $\mathcal{B}$ is called \emph{power associative} if the subalgebras generated by single elements of $\mathcal{B}$ are associative. In the case of a Jordan algebra $M$ this is equivalent to say that the identity $ a^m\circ a^n=a^{m+n},$ holds for all $a\in M$, $m,n\in \NN.$ It is known that any Jordan algebra is power associative (\cite[Lemma 2.4.5]{HOS}).
\smallskip

From a purely algebraic point of view, elements $a, b$ in a Jordan algebra $M$ are said to \emph{operator commute} if $$(a\circ c)\circ b= a\circ (c\circ b),$$ for all $c\in M$ (cf. \cite[4.2.4]{HOS}). The \emph{centre} of $M$ is, by definition, the set of all elements of $M$ which operator commute with any other element in $M$, and will be denoted by $Z(M)$. Elements in the centre of $M$ are called \emph{central}. \smallskip

An element $a$ in a unital Jordan Banach algebra $M$ is called \emph{invertible} if there exists $b\in M$ satisfying $a \circ b = \textbf{1}$ and $a^2 \circ b = a.$ The element $b$ is unique and it will be denoted by $a^{-1}$ (cf. \cite[3.2.9]{HOS} and \cite[Definition 4.1.2]{Cabrera-Rodriguez-vol1}). We know from \cite[Theorem 4.1.3]{Cabrera-Rodriguez-vol1} that an element $a\in M$ is invertible if and only if $U_a$ is a bijective mapping, and in such a case $U_a^{-1} = U_{a^{-1}}$. The set $M^{-1}$ of all invertible elements in $M$ is open (see \cite[Theorem 4.1.7]{Cabrera-Rodriguez-vol1}), and hence $M^{-1}$ is locally connected and its connected components are open. An element of the form $U_x (y)$ is invertible if and only if both $x$ and $y$ are (cf. \cite[Theorem 4.1.3$(vi)$]{Cabrera-Rodriguez-vol1}), and then $(U_x (y))^{-1} = U_x^{-1} (y^{-1})$; consequently, $(x^{-1})^2 = (x^2)^{-1} = x^{-2}$. The reader should be warned that the Jordan product of two invertible elements is not, in general, an invertible element. However, the square and the $n$-th power ($n\in \mathbb{N}$) of each invertible element is an invertible element.
\smallskip

The spectrum of an element $a$ of a Jordan-Banach algebra $M$, $Sp(a)$, is the set of all $\lambda\in \mathbb{C}$ for which $a-\lambda \textbf{1}$ is not invertible. As in the setting of associative Banach algebras, $Sp(a)$ is a non-empty compact subset of the complex plane (see \cite[Theorem 4.1.17]{Cabrera-Rodriguez-vol1} where even a Jordan version of the famous Gelfand-Beurling formula is stated). Moreover the mapping $a\mapsto Sp(a)$ is upper semi-continuous on $M$. The symbol $\rho(a)$ will stand for the spectral radius of $a$, that is, the greatest modulus of spectral values.\smallskip

An element $a$ in a unital Jordan-Banach algebra $M$ is \emph{quasi-invertible} with \emph{quasi-inverse} $b \in M$ if $\textbf{1} -a$ is invertible with inverse $\textbf{1} - b$. A subset of $M$ is called quasi-invertible if all its elements are quasi-invertible in $M$.\smallskip

An \emph{outer ideal} of a Jordan-Banach algebra $M$ is a subspace $J$ such that $U_{M} (J)\subseteq J$; an \emph{inner ideal} or a \emph{strict inner ideal} is a subspace $I$ satisfying $U_{I} (M) \subseteq I$ and $I^2 \subseteq I$ (if $M$ is unital the second condition clearly follows from the first one). An ideal $J$ is a subspace that is both an outer and a strict inner ideal (i.e. $J^2 \subseteq J$). It is known that if $J$ is an ideal or an inner ideal, then the quasi-inverse $w$ of any quasi-invertible element $z\in J$ also belongs to $J$ (see \cite[page 672]{McCrimm69}).\smallskip

Let us recall the notion of radical in the Jordan setting. For each Jordan-Banach algebra $M$ there exists a unique maximum quasi-invertible ideal $\mathcal{R}ad(M)$ of $M,$ called the \emph{McCrimmon radical} of $M$, which contains all quasi-invertible ideals (see \cite[Theorem 1]{McCrimm69}). In \cite{Cabrera-Rodriguez-vol1} and in the original paper by K. McCrimmon \cite{McCrimm69}, the McCrimmon radical is called the \emph{Jacobson radical} (cf. \cite[Proposition 4.4.11 and Definition 4.4.12]{Cabrera-Rodriguez-vol1}). If $M$ is a special Jordan-Banach algebra the McCrimmon radical coincides with the usual Jacobson radical (cf. \cite[Theorem 3.6.21]{Cabrera-Rodriguez-vol1}). L. Hogben and K. McCrimmon proved in \cite[Theorem 1.1]{HogbMcCrimm81} that the McCrimmon radical of a unital Jordan algebra coincides with the intersection of all maximal inner ideals. A Jordan-Banach algebra $M$ is called \emph{Jordan-semisimple} or simply \emph{J-semisimple} if $\mathcal{R}ad(M)= \{0\}$.\smallskip

B. Aupetit established in \cite{Aupetit93} several characterization of the Jacobson and McCrimmon radicals of associative Banach algebras and Jordan-Banach algebras, respectively. The next result has been borrowed from the just quoted reference.

\begin{theorem}\label{t Aupetit c 1}{\rm \cite[Theorem 2, Corollaries 1 and 2]{Aupetit93}} Let $a$ be an element of a Jordan-Banach algebra $M$. Then the following statements are equivalent:\begin{enumerate}[$(a)$]\item $a$ is in the McCrimmon radical of $M$;
\item $\sup\{\rho(x+\alpha a) : \alpha\in \mathbb{C}\}<\infty$, for every $x\in M$;
\item $\rho(U_x(a)) = 0,$ for every $x$ in $M$;
\item There exists $C\geq 0$ such that $\rho (x) \leq C \|x-a\|$ for all $x$ in a neighborhood of $a$.
\end{enumerate}
\end{theorem}

Let $A$ be a unital associative Banach algebra with unit $\textbf{1}$, whose subgroup of invertible elements is denoted by $A^{-1}$. Let $A_{\textbf{1}}^{-1}$ denote the principal component of $A^{-1}$. The next result was proved by O. Hatori in \cite{Hat2011}.

\begin{lemma}\label{l Hatori associative}{\rm\cite[Lemma 3.1]{Hat2011}} Let $A$ be a unital (associative) Banach algebra and $a\in A$. Suppose that $\rho(ba) = 0$ for every $b$ in the principal component of $A^{-1}$. Then $a$ lies in the Jacobson radical of $A$.
\end{lemma}

Our next goal is a Jordan version of the previous result. As we have already commented, the left and right multiplication operations have no direct Jordan analogue, and the set of invertible elements in a Jordan-Banach algebra $M$ is not, in general, closed under Jordan products; however $U_a(b)$ is invertible if and only if $a$ and $b$ are (cf. \cite[Theorem 4.1.3$(vi)$]{Cabrera-Rodriguez-vol1}).\smallskip

As pointed out by Aupetit in \cite[Theorem 2.3]{Aupetit95}, since the closed subalgebra of a unital Jordan-Banach algebra $M$ generated by $\textbf{1}$ and an element $a$ is an associative Banach algebra, the standard holomorphic functional calculus in Banach algebras can be extended to the situation of Jordan-Banach algebras. Beside the properties stated in \cite[Theorem 2.3]{Aupetit95} we shall employ the detailed construction developed in \cite[Theorems 4.1.88 and 4.1.93]{Cabrera-Rodriguez-vol1}. Suppose $\Omega$ is an open neighborhood of Sp$(a)$, and let $h_0$ stand for the inclusion mapping $\Omega\hookrightarrow \mathbb{C}$. Then there is a unique continuous unit-preserving algebra homomorphism $f \mapsto f (a)$ from the complex algebra $\mathcal{H} (\Omega),$ of all complex-valued holomorphic functions on $\Omega,$ into $M$ taking $h_0$ to $a$ and satisfying the following properties:\begin{enumerate}[$(i)$]\item $\displaystyle f(a) = \frac{1}{2 \pi i} \int_{\Gamma} f(\lambda) (\lambda \textbf{1}-a)^{-1} d\lambda$, for any positively orientated curve included in $\Omega$ and surrounding Sp$(a)$;
\item $f(a)$ is contained in the smallest closed strongly associative subalgebra of $M$ containing $\textbf{1}$ and $a$;
\item (Spectral mapping theorem) For each $f\in \mathcal{H} (\Omega)$ we have
$\hbox{Sp}(M, f (a)) = f (\hbox{Sp}(M,a))$;
\item The set $$M_{\Omega} : =\{ x\in M : \hbox{Sp} (M,x) \subseteq \Omega\}$$ is a non-empty open subset of $M$, and the mapping  $\widetilde{f} : x \mapsto f (x)$ from $M_{\Omega}$ to $M$ is holomorphic (i.e. there exists a bounded linear operator $T: M\to M$ such that $\displaystyle \lim_{h\to 0} \frac{\|f(x + h) - f(x)- T(h)\|}{\|h\|} = 0$ for any $x\in M_{\Omega}$).
\end{enumerate}

Henceforth, given a Banach space $X$, the open (respectively, closed) ball of radius $\rho$ and center $a\in X$ will be denoted by $B_{_X}(a,\rho)$ or by $B(a,\rho)$ (respectively, $\overline{B}_{_X}(a,\rho)$).

\begin{proposition}\label{p Hatori-Aupetit Jordan} Let $M$ be a unital Jordan-Banach algebra and $a\in M$. Then the following statements are equivalent:
\begin{enumerate}[$(a)$]\item $a$ is in the McCrimmon radical of $M$;
\item $\rho(U_x(a)) = 0,$ for every $x$ in $M$;
\item $\rho(U_b(a)) = 0,$ for every $b$ in $M^{-1}$;
\item $\rho(U_b(a)) = 0,$ for every $b$ in the principal component $M_{\mathbf{1}}^{-1}$ of $M^{-1}$.
\end{enumerate}
\end{proposition}

\begin{proof} The equivalence $(a) \Leftrightarrow (b)$ is given by Theorem \ref{t Aupetit c 1}. The implications $(b)\Rightarrow (c)\Rightarrow (d)$ are clear.\smallskip

To prove the implication $(d)\Rightarrow (a)$ we shall show how to adapt and modify an argument from Aupetit's paper \cite{Aupetit93}. Fix an arbitrary $y\in M$ and $\mu \in \mathbb{C}$ with $|\mu|>\rho (y).$ Let us observe that $y - t \frac{\mu}{|\mu|} \textbf{1}\in M^{-1},$ equivalently, $- \frac{|\mu|}{\mu t} y +\textbf{1}\in M^{-1},$ for all $t\in [|\mu|, +\infty)$. The mapping $\gamma: [0,1]\to M^{-1}$, $\gamma(s) =- s \mu^{-1} y +\textbf{1}$ gives a continuous curve connecting $\textbf{1}$ and $-\mu^{-1} y +\textbf{1}$. Since for each $\alpha\in \mathbb{C}\backslash\{0\}$ we can easily find a continuous path in $M^{-1}$ connecting $\textbf{1}$ and $\alpha \textbf{1}$ (actually if $a\in M^{-1}$ is connected by a continuous path in $M^{-1}$ with $\textbf{1}$, the same holds for $\alpha a$), it follows that $ y - t \frac{\mu}{|\mu|} \textbf{1}\in M_{\mathbf{1}}^{-1},$ for all $t\in [|\mu|, +\infty)$.\smallskip

Let us make some observations on the spectrum of the elements of the form $ y - t \frac{\mu}{|\mu|} \textbf{1}\in M_{\mathbf{1}}^{-1},$ with $t\in [|\mu|, +\infty)$. 
 Since for $t\in [|\mu|, +\infty)$ we have $$\hbox{Sp}\left(y - t \frac{\mu}{|\mu|} \textbf{1}\right) =\hbox{Sp}(y)- t \frac{\mu}{|\mu|}\subseteq \rho(y) \overline{B}_{\mathbb{C}} (0,1) -t \frac{\mu}{|\mu|}= \overline{B}_{\mathbb{C}} (-t \frac{\mu}{|\mu|}, \rho(y)) $$ and $$\hbox{Sp}\left(\textbf{1} -  \frac{|\mu|}{t} \mu^{-1} y \right) = -  \frac{|\mu|}{t} \mu^{-1}  \hbox{Sp}\left(y - t \frac{\mu}{|\mu|} \textbf{1}\right) \subseteq \{z\in \mathbb{C} : \Re\hbox{e}(z) >0\}= \Omega.$$\smallskip

Consider the holomorphic function $f: \Omega \to \mathbb{C},$ $f(\lambda) = \frac{1}{\sqrt{\lambda}}$ --where we consider the principal branch of the square root-- and the holomorphic functional calculus to compute $\left(\textbf{1} -  \frac{|\mu|}{t} \mu^{-1} y \right)^{-\frac12} = f \left(\textbf{1} -  \frac{|\mu|}{t} \mu^{-1} y \right)$ with $t\geq |\mu|$. By the continuity of the holomorphic functional calculus, the mapping $\gamma : [|\mu|, \infty)\to M^{-1},$ $\gamma(t) = \left(\textbf{1} -  \frac{|\mu|}{t} \mu^{-1} y \right)^{-\frac12}$ is continuous, and clearly satisfies $\gamma (|\mu|) = \left(\textbf{1} -  \mu^{-1} y \right)^{-\frac12}$ and $\lim_{t\to +\infty} \gamma (t) = \textbf{1}$. Therefore the element $x= \left(\textbf{1} -  \mu^{-1} y \right)^{-\frac12}$ lies in $M_{\mathbf{1}}^{-1},$ and hence $\beta x\in M_{\mathbf{1}}^{-1}$ for all $\beta \in \mathbb{C}\backslash\{0\}$.\smallskip

We know from the hypothesis that $\rho (U_{\beta x} (a))=0$ for all non-zero $\beta\in \mathbb{C}$. In this case $$U_{x} (\textbf{1} -  \mu^{-1} y- \beta^2 a) = \textbf{1} - U_{\beta x} (a) \hbox{ is an invertible element for all } \beta\in \mathbb{C}\backslash\{0\}.$$ But the same conclusion trivially holds for $\beta =0$. We therefore deduce that the element $U_{x} (\textbf{1} -  \mu^{-1} y+ \alpha a)$ is invertible for all $\alpha\in \mathbb{C}$, and by the fundamental identity $$U_{U_{x} (\textbf{1} -  \mu^{-1} y+\alpha a)} = U_{x} U_{\textbf{1} -  \mu^{-1} y+ \alpha a} U_{x}, \hbox{ and } \textbf{1} -  \mu^{-1} y +\alpha a \hbox{ both are invertible } \forall \alpha\in \mathbb{C}.$$ In particular, $$ - \mu \textbf{1} +  y -\alpha a = - \mu ( \textbf{1} -  \mu^{-1} y +\alpha  \mu^{-1} a )\in M^{-1},$$ for all $\alpha \in \mathbb{C}$. We thus deduce that $\mu \notin Sp(y -\alpha a)$ for all $\alpha \in \mathbb{C}$. This proves that for each $\varepsilon >0$, $$\mathbb{C}\backslash B_{\mathbb{C}}( 0, \rho(y) + \varepsilon) \subseteq \mathbb{C}\backslash Sp(y -\alpha a),$$ equivalently,
$Sp(y -\alpha a)  \subseteq B_{\mathbb{C}}( 0, \rho(y) + \varepsilon),$ and thus $\rho (y -\alpha a) \leq \rho(y) + \varepsilon,$ for all $\alpha\in \mathbb{C}.$ We can clearly let $\varepsilon \to 0$. It can be easily deduced that for each $z$ in $M$ taking $y = z-a$ we have $$\rho (z) = \rho ( y +a ) \leq \rho (y)\leq \|y\| = \| z-a\|,$$ and Aupetit's Theorem \ref{t Aupetit c 1} implies that $a$ lies in the McCrimmon radical of $M$.
\end{proof}

It is perhaps worth to remark several of the properties employed in the proof above. The reader will probably feel more confortable with some references in the Jordan setting. A more detailed exposition is presented in the next remark.

\begin{remark}\label{r properties of the identity component Jordan} Let $M$ be a unital Jordan-Banach algebra and let $M_{\mathbf{1}}^{-1}$ denote the principal component of $M^{-1}$. Then the following properties hold: \begin{enumerate}[$(1)$]\item If $b\in M_{\mathbf{1}}^{-1}$ and $\alpha\in \mathbb{C}\backslash\{0\},$ the element $\alpha b$ lies in $M_{\mathbf{1}}^{-1}$;
\item If $y\in M$ and $\mu \in \mathbb{C}$ with $|\mu|>\rho (y),$ the elements $(y \pm \mu \textbf{1})$ and $(y \pm \mu \textbf{1})^{-2}$ belong to $M_{\mathbf{1}}^{-1};$
\item If $a\in M^{-1}$ then its square is also in $M^{-1}$. Furthermore, defining $a^{-n} := (a^{-1})^n$ for $n > 0$ and $a^0 := \textbf{1}$, then $a^k \circ a^l = a^{k+l}$ for all integer numbers $k,l$. The element $a^n$ is invertible for every natural $n,$ and $(a^{n})^{-1} = (a^{-1})^n$ for all integer number $n$;
\item The inverse of every element in the principal component $M_{\mathbf{1}}^{-1}$ lies in $M_{\mathbf{1}}^{-1}$.
\end{enumerate}

Statements $(1)$ and $(2)$ have been explicitly justified in the proof of the previous result. Statement $(3)$ is proved in \cite[Theorem 4.1.3$(v)$]{Cabrera-Rodriguez-vol1}. The final statement follows from the fact that the mapping $M^{-1}\to M^{-1}$, $a\mapsto a^{-1}$
is a homeomorphism {\rm(}cf. \cite[Proposition 4.1.6]{Cabrera-Rodriguez-vol1}{\rm)} with $\textbf{1} = \textbf{1}^{-1}.$
\end{remark}

In order to review some additional properties of the principal component of the set of invertible elements in a Jordan-Banach algebra $M$, we shall recall the exponential function in this setting. For each element $a$ in $M$, the series $\displaystyle \exp(a) =\sum_{n=0}^\infty \frac{a^n}{n!}$ is uniformly convergent on bounded subsets of $M$. Furthermore, the mapping $a\mapsto \exp (a)$ defines an analytic mapping on $M$. It is known that the Jordan-Banach subalgebra $C$ of $M$ generated by an element $a$ is a commutative associative subalgebra with respect to the inherited Jordan product \cite[1.1]{Jacobson81}. If an associative Banach algebra $A$ is regarded as a Jordan-Banach algebra, for each $a\in A$, $\exp(a)$ is just the usual
exponential series in the usual sense for associative Banach algebras.  It follows that $\exp(a)$ has its usual meaning in $C$, and consequently $\exp(s a )\  \exp(t a) = \exp((s + t) a)$
for all $s, t \in \mathbb{C}$. We also have $U_a (b) = a^2 \circ b$ for $a, b \in  C$, which yields $$U_{\exp( t x)} (\exp(-tx)) = \exp(tx),\hbox{ and } U_{\exp(tx)} (\exp( -tx)^2) = \textbf{1}.$$ It follows, for example from \cite[1.6.1]{Jacobson81}, that $\exp(tx)$ is invertible in $M$ with inverse $\exp(-tx)$.\smallskip

A well known result in the theory of associative Banach algebras proves that for each unital Banach algebra $A$, the least subgroup of $A^{-1}$ containing $\exp(A)$ is the principal component of $A^{-1}$ (cf. \cite[Propositions 8.6 and 8.7]{BonDunCNA73}). The description of the principal component of the invertible elements in a Jordan-Banach algebra was an open problem for many years. The final answer was obtained by O. Loos in \cite{Loos96}, where it is shown that for each unital Jordan-Banach algebra $M$ the principal component of $M^{-1}$ is the set $$M_{\mathbf{1}}^{-1} = \left\{ U_{\exp(a_1)}\cdots U_{\exp(a_n)} (\textbf{1}) : a_i\in M, \ n\geq 1 \right\},$$ in particular, $M_{\mathbf{1}}^{-1}$ is open and closed in $M^{-1}$. It follows from this that each connected component of $M^{-1}$ is analytically arcwise connected \cite[Corollary]{Loos96}. It should be noted that in Loos's theorem the norm is not required to satisfy $\|x\circ y\|\leq \|x\| \ \|y\|$ and $\|\textbf{1}\|=1$.\smallskip

The subset of invertible elements in a unital Jordan-Banach algebra $M$ lacks stability under the Jordan product. However, we have $$U_{M^{-1}} (M^{-1}) = \{ U_a (b) : a,b\in M^{-1}\}\subseteq M^{-1}.$$ This motivates us to introduce the following notion. A subset $\mathfrak{M}$ of $M^{-1}$ will be called a \emph{quadratic subset} if $U_{\mathfrak{M}} (\mathfrak{M}) \subseteq \mathfrak{M}$.

\begin{lemma}\label{l principal component as the least quadratic subset of invertible elements} Let $M$ be a unital Jordan-Banach algebra. Then the principal component of $M^{-1}$ is precisely the least quadratic subset of $M^{-1}$ containing $\exp(M)$. As a consequence, for each integer number $n$ and each $a\in M_{\mathbf{1}}^{-1}$ the element $a^{n}$ lies in $M_{\mathbf{1}}^{-1}$.
\end{lemma}

\begin{proof} Let us first show that $M_{\mathbf{1}}^{-1}$ is a quadratic subset of $M^{-1}$ (clearly, $M_{\mathbf{1}}^{-1}$ contains $\exp(M)$). For this purpose, by Loos' theorem \cite{Loos96}, we take two elements $a = U_{\exp(a_1)}\cdots U_{\exp(a_n)} (\textbf{1})$ and $b = U_{\exp(b_1)}\cdots U_{\exp(b_m)} (\textbf{1})$ in $M_{\mathbf{1}}^{-1}$. By the fundamental formula $$U_{a} (b) = U_{U_{\exp(a_1)}\cdots U_{\exp(a_n)} (\textbf{1})} (U_{\exp(b_1)}\cdots U_{\exp(b_m)} (\textbf{1})) $$
$$ = U_{\exp(a_1)} \cdots U_{\exp(a_n)} U_{\textbf{1}} U_{\exp(a_n)}\cdots U_{\exp(a_1)} U_{\exp(b_1)}\cdots U_{\exp(b_m)} (\textbf{1}) \in M_{\mathbf{1}}^{-1}.$$ Given $a\in M_{\mathbf{1}}^{-1}$, it follows from the above that $a^2 = U_{a} (\textbf{1})\in M_{\mathbf{1}}^{-1}$, $a^3 = U_{a} (a)\in M_{\mathbf{1}}^{-1}$, 
and by an induction argument $a^{n} = U_{a} (a^{n-2})\in M_{\mathbf{1}}^{-1}$ for all $n\geq 3$. We have seen in Remark \ref{r properties of the identity component Jordan}$(3)$ that the inverse of every element in $M_{\mathbf{1}}^{-1}$ also lies in $M_{\mathbf{1}}^{-1}$. Therefore the previous arguments also show that $a^{n}\in M_{\mathbf{1}}^{-1}$ for all $a\in M_{\mathbf{1}}^{-1}$ and every integer number $n$. \smallskip

Suppose $\mathfrak{M}$ is a \emph{quadratic subset} of $M^{-1}$ containing $\exp(M)$. Since, clearly $\textbf{1}\in \mathfrak{M}$, we have $U_{\exp(a)} (\textbf{1})\in \mathfrak{M}$. Suppose that $U_{\exp(a_n)}\cdots U_{\exp(a_1)} (\textbf{1})\in \mathfrak{M}$ for $n\geq 2$, $a_i\in M$.  Then, by the assumption on $\mathfrak{M}$, we have $$U_{\exp(a_{n+1})} U_{\exp(a_n)}\cdots U_{\exp(a_1)} (\textbf{1}) \in U_{\mathfrak{M}} (\mathfrak{M}) \subseteq \mathfrak{M}.$$ It follows by induction and Loos' theorem that $M_{\mathbf{1}}^{-1} \subseteq \mathfrak{M}.$
\end{proof}

\begin{remark}\label{r open subgroup} Let $A$ be a unital (associative) Banach algebra. Let $\mathfrak{A}$ be an open subgroup of $A^{-1}$. It is well known that under these assumptions $\mathfrak{A}$ is closed, and consequently $\mathfrak{A}$ contains the principal component of $A^{-1}$. The proof of this result employs the left or right multiplication operations in $A$, operations which have no direct Jordan analogue. For a Jordan-Banach algebra $M$, it would be interesting to know if every open quadratic subset $\mathfrak{M}$ of $M^{-1}$ containing the unit must be also closed, and hence $M_{\mathbf{1}}^{-1}\subseteq \mathfrak{M}$.\smallskip

Furthermore, since $A^{-1}$ is open, and hence locally connected and its connected component are open, and $\mathfrak{A}$ is clopen a classic result in topology {\rm(}see, for example, \cite[Theorem 5]{Niem77}{\rm)} asserts that $\mathfrak{A}$ must coincide with the union of some connected components of $A^{-1}$ {\rm(}but not necessarily all of them{\rm)}.\smallskip

The main component $A_{\textbf{1}}^{-1}$ of $A^{-1}$ is a subgroup of $A$ {\rm(}see \cite[Proposition 8.6]{BonDunCNA73}{\rm)}. If $\mathcal{C}$ is another connected component of $A^{-1}$ and we pick $b\in \mathcal{C}$, it is easy to see that $\mathcal{C} = b A^{-1}_{\textbf{1}} = A_{\textbf{1}}^{-1} b$ {\rm(}just apply that the left and right multiplication operators by $b$ are homeomorphisms mapping $1$ to $b${\rm)}. The component $\mathcal{C}$ need not be a subgroup, however, if we replace the product of $A$ with the one defined by $x \cdot_{b^{-1}} y = x b^{-1} y$, we get another another associative complete normed algebra $A_{b^{-1}}$, called the \emph{$b^{-1}$-homotope} of $A$, with unit $b$. The main connected component of $A_{b^{-1}}$ is precisely $\mathcal{C}$, which is therefore a subgroup of $A_{b^{-1}}$.
\end{remark}

\section{Surjective isometries between subsets of invertible elements in Jordan-Banach algebras}\label{sec: surjective isometries between subsets of invertible elements in Jordan-Banach algebras}

The lacking of left and right multiplication operators in the Jordan setting increase the difficulties to get similar conclusions to those in the previous Remark \ref{r open subgroup} for Jordan-Banach algebras. As in previous contributions, passing to an homotope is in some sense a
substitute for the left multiplication by an invertible element in an associative Banach algebra. We recall the notion of \emph{isotope} \cite[1.7]{Jacobson81} (also called \emph{homotope} by McCrimmon in \cite{McCrimm71}): Suppose $c$ is an invertible element in a unital Jordan-Banach algebra $M$ then the vector space $M$ becomes a Jordan-Banach algebra $M_c$ with unit element $c^{-1}$, Jordan product $a\circ_c b := U_{a,b} (c)$, and quadratic operators $U^{(c)}_{a} = U_a U_c$, for all $a,b\in M$.\smallskip

Henceforth we shall consider two unital Jordan-Banach algebras $M$ and $N$. Let $M_{\mathbf{1}}^{-1}$ and $N^{-1}_{\mathbf{1}}$ denote the principal components of $M^{-1}$ and $N^{-1}$, respectively. We shall consider two clopen subsets $\mathfrak{M}\subseteq M^{-1}$ and $\mathfrak{N}\subseteq N^{-1}$ which are quadratic subsets of $M^{-1}$ and $N^{-1}$, respectively, and are closed for powers and inverses. In particular $\mathfrak{M}$ and $\mathfrak{N}$ coincide with the union of some connected components of $M^{-1}$ and $N^{-1}$, respectively. It is easy to see that $\mathfrak{M}$ (respectively, $\mathfrak{N}$) contains the unit of $M$ (respectively, the unit of $N$). Namely, given $b\in \mathfrak{M}$, $b^{-1}$, $b^{-2}$ and $1 = U_b (b^{-2})$ belong to $\mathfrak{M}$. By applying that $\mathfrak{M}$ and $\mathfrak{N}$ are clopen sets we obtain $M_{\mathbf{1}}^{-1}\subseteq \mathfrak{M}$ and $N_{\mathbf{1}}^{-1} \subseteq \mathfrak{N}$. We observe that the sets $M_{\mathbf{1}}^{-1}= \mathfrak{M}$ and $N_{\mathbf{1}}^{-1} = \mathfrak{N}$ satisfy the stated hypotheses.\smallskip

As in Remark \ref{r properties of the identity component Jordan}$(1)$ if $a$ is an element in a connected component, $\mathcal{C}$, of $M^{-1}$ and $\alpha\in \mathbb{C}\backslash\{0\},$ then $\alpha a\in  \mathcal{C}$. In particular $(\mathbb{C}\backslash\{0\}) \mathfrak{M} = \mathfrak{M}$.\smallskip

Our goal is to determine the structure of every surjective isometric mapping $\Delta: \mathfrak{M}\to \mathfrak{N}$. The result will be derived after a series of lemmata and propositions.\smallskip

Let $\mathcal{D}$ be a convex subset of a Banach space $X$. We shall say that a mapping $F$ from $\mathcal{D}$ into another Banach space $Y$ is \emph{real affine} or simply \emph{affine} if for each $t\in [0,1]$, and $x,y\in \mathcal{D}$ we have $F(t x + (1-t) y ) = t F(x) + (1-t) F(y)$. We begin with an observation on the local affine behavior of $\Delta$. We state next an easy observation, which combined with a powerful result of Mankiewicz asserting that every surjective isometry between convex bodies in two arbitrary normed spaces can be uniquely extended to an affine function between the spaces (see \cite[Theorem 5]{Mank1972}), provides an useful tool for our goals.

\begin{lemma}\label{l Delta is locally affine} Let $\Delta : \mathfrak{M}\to \mathfrak{N}$ be a surjective isometry. Then $\Delta$ is a local affine mapping, concretely, for each $a\in \mathfrak{M}$ there exists a positive $\delta_a$, depending on $a$, such that $B(a,\delta_a)\subseteq \mathfrak{M}$, $B(\Delta(a), \delta_a)\subseteq \mathfrak{N}$, $\Delta(B(a,\delta_a)) = B(\Delta(a), \delta_a)$ and there exists a surjective affine isometry $F_{a, \delta_a} : M\to N$ such that $\Delta|_{B(a,\delta_a)} = F_{a, \delta_a}|_{B(a,\delta_a)}$ {\rm(}is an affine mapping{\rm)}. Furthermore, suppose that $\gamma : [0,1]\to \mathfrak{M}$ is a continuous path. Then there exists a surjective affine isometry $F: M\to N$ and an open neighborhood $U$ of $\gamma([0,1])$ such that $U\subseteq \mathfrak{M}$ and $F|_{U} = \Delta$.
\end{lemma}

\begin{proof} Since $\Delta$ is a surjective isometry, for each $\delta >0$, $a\in \mathfrak{M}$, we have $$\Delta\left( B(a,\delta) \cap \mathfrak{M} \right) =B(\Delta(a),\delta) \cap \mathfrak{N}.$$ Since $\mathfrak{M}$ and $\mathfrak{N}$ are open we can find $\delta_a>0$ satisfying $B(a,\delta_a)\subset \mathfrak{M}$, $B(\Delta(a), \delta_a)\subset \mathfrak{N}$ and $\Delta(B(a,\delta_a)) \subseteq B(\Delta(a), \delta_a)$. Since $\Delta|_{B(a,\delta_a)}: B(a,\delta_a) \to B(\Delta(a), \delta_a)$ is a surjective isometry, the final conclusion follows from a celebrated result due to P. Mankiewicz (see \cite[Theorem 2]{Mank1972}).\smallskip

For the final statement, let us take a continuous path $\gamma : [0,1]\to \mathfrak{M}$. For each $t\in [0,1],$ by the first part of the proof, there exist $\delta_t>0$ and a surjective affine isometry $F_{t} : M\to N$ satisfying $B(\gamma(t),\delta_t)\subseteq \mathfrak{M}$, $B(\Delta(\gamma(t)), \delta_t)\subseteq \mathfrak{N}$, $\Delta(B(\gamma(t),\delta_t)) = B(\Delta(\gamma(t)), \delta_t),$ and $\Delta|_{B(\gamma(t),\delta_t)} = F_{t}|_{B(\gamma(t),\delta_t)}$. By applying that $\gamma([0,1])$ is compact we deduce the existence of $t_0=0<t_1<\ldots<t_{n-1}<t_n =1$ such that $\displaystyle \gamma ([0,1]) \subseteq \bigcup_{k=0}^n B(\gamma(t_{k}),\delta_{t_k})$. We claim that $F_{t_{k}} = F_{t_{l}}$ for all $k,l\in \{0,\ldots,n\}$. We observe that two surjective affine isometries from $M$ onto $N$ coinciding on an open subset must be the same. If the open set $B(\gamma(t_{k}),\delta_{t_k})\cap B(\gamma(t_{l}),\delta_{t_l})$ is non-empty we clearly have $F_{t_{k}} = F_{t_{l}}$, otherwise, by the connectedness of $\gamma([0,1])$, we can find a finite collection $t_k= t_{j_1},\ldots,t_{j_m}=t_{l}\in \{t_0=0,t_1,\ldots,t_{n-1},t_n\}$ such that $$B(\gamma(t_{j_k}),\delta_{t_{j_k}})\cap B(\gamma(t_{j_{k+1}}),\delta_{t_{j_{k+1}}})\neq \emptyset,$$ for all $k\in \{1,\ldots, m-1\}$. It then follows that $F_{t_{k}} = F_{t_{j_1}} = F_{t_{j_{2}}} = \ldots= F_{t_{j_m}} = F_{t_{l}}$, which gives the desired statement. Therefore the mapping $F= F_{t_0}: M\to N$ is a surjective affine isometry, $\displaystyle U = \bigcup_{k=0}^n B(\gamma(t_{k}),\delta_{t_k})$ is an open neighborhood of $ \gamma ([0,1])$ and $F|_{U} = \Delta|_{U}$ because $\Delta|_{B(\gamma(t_{k}),\delta_{t_k})} = F_{t_{k}}|_{B(\gamma(t_{k}),\delta_{t_k})} = F|_{B(\gamma(t_{k}),\delta_{t_k})}$ for all $0\leq k\leq n$.
\end{proof}

We continue with a proposition showing that, as in the case of associative Banach algebras \cite{Hat2011}, each surjective isometry between the sets $\mathfrak{M}$ and $\mathfrak{N}$ admits a limit in zero and this limit is precisely an element in the McCrimmon radical of the Jordan-Banach algebra $N$.

\begin{proposition}\label{p 1} Let $\Delta : \mathfrak{M}\to \mathfrak{N}$ be a surjective isometry. Then the limit $\displaystyle \lim_{\mathfrak{M}\ni a\to 0} \Delta(a)$ exists, and its value is an element $u_0$ in the McCrimmon radical of $N$.
\end{proposition}

\begin{proof} Let $(a_n)_n$ be a sequence in $\mathfrak{M}$ converging to $0$. By applying that $\Delta$ is an isometry, we can easily prove that $(\Delta(a_n))_n$ is a Cauchy sequence in $N$, therefore there exists $u_a\in N$, depending on $(a_n)_n$, such that $(\Delta(a_n))_n\to u_a$. If $(b_n)_n$ is any other null sequence in $\mathfrak{M}$, we can similarly prove that $(\Delta(b_n))_n$ converges to some $u_b\in N$. The hypotheses on $\Delta$ assure that $$\| \Delta(a_n)- \Delta(b_n)\| = \|a_n-b_n\|\to 0,$$ and hence $u_a = u_b$.\smallskip

We shall finally prove that $u_0\in \mathcal{R}ad(N)$. Since $\mathfrak{M}$ is a quadratic subset of $M^{-1}$ and $\exp(M)\subset \mathfrak{M}$, given $c\in \mathfrak{M}$ and $\alpha\in \mathbb{C}\backslash\{0\}$, by taking $\beta\in \mathbb{C}\backslash\{0\}$ with $\beta^2 =\alpha$ we deduce that  $\alpha c = U_{\beta 1} (c) \in \mathfrak{M}$.\smallskip

Fix an arbitrary $b\in N_{\mathbf{1}}^{-1}\subseteq \mathfrak{N}$ and $\lambda \in Sp(U_b(u_0))\backslash\{0\}$. Since $-\lambda b^{-2}\in N_{\mathbf{1}}^{-1}\subseteq \mathfrak{N}$ (cf. Lemma \ref{l principal component as the least quadratic subset of invertible elements} and Remark \ref{r properties of the identity component Jordan}), by the surjectivity of $\Delta$ there exists (a unique) $c_{\lambda}\in \mathfrak{M}$ such that $\Delta(c_{\lambda}) = - \lambda b^{-2}$.\smallskip

As we have mentioned several times before, $t c_{\lambda}\in \mathfrak{M}$ for every $t \in \mathbb{R}\backslash\{0\}$. Fix an arbitrary $0<s<\frac12$ in $\mathbb{R}$, by applying Lemma \ref{l Delta is locally affine} to the continuous path $\gamma(t) = t s c_{\lambda} + (1-t) (1-s) c_{\lambda}$, we can find an open neighborhood $U$ of $\gamma([0,1])$ with $U \subseteq \mathfrak{M}$ and a surjective affine isometry $F: M\to N$ such that $\Delta|_{U} = F|_{U}$. Since $F$ is affine on $U$ and $\gamma([0,1])\subseteq U$ we deduce that $$\begin{aligned}\Delta \left(\frac{c_{\lambda}}{2}\right) &= F \left(\frac12 s c_{\lambda} + \frac12 (1-s) c_{\lambda}\right) = \frac{F((1-s) c_{\lambda}) +F (s c_{\lambda})}{2} \\
&= \frac{\Delta((1-s) c_{\lambda}) +\Delta (s c_{\lambda})}{2}.
\end{aligned}$$

Since $s$ was an arbitrary real number in $(0,\frac12)$, by taking $\displaystyle \lim_{s\to 0^+}$ in the above identity and applying that $\Delta$ is continuous and the first part of this proof, we get $$\Delta \left(\frac{c_{\lambda}}{2}\right) = \frac{\Delta(c_{\lambda}) + u_0}{2} = \frac{- \lambda b^{-2} + u_0}{2}.$$ Evaluating the mapping $U_b$ at both sides of the previous equality, and having in mind that $\mathfrak{N}$ is a quadratic set, we deduce that $$ - \lambda \textbf{1} + U_{b} (u_0) = -\lambda U_b (b^{-2}) + U_{b} (u_0) = 2 U_b \left( \Delta \left(\frac{c_{\lambda}}{2}\right)\right) \in U_{\mathfrak{N}} (\mathfrak{N})\subseteq \mathfrak{N}\subseteq N^{-1},$$ witnessing that $\lambda\notin Sp (U_b(u_0))$ and leading to a contradiction. We have therefore shown that $Sp(U_b(u_0))=\{0\}$. The arbitrariness of $b\in N_{\mathbf{1}}^{-1}\subseteq \mathfrak{N}$ and Proposition \ref{p Hatori-Aupetit Jordan} guarantee that $u_0$ belongs to $\mathcal{R}ad (N)$.\end{proof}

\begin{lemma}\label{l N-u0 in N} Let $N$ be a unital Jordan-Banach algebra, and let $\mathfrak{N}\subseteq N^{-1}$ be a clopen subset which is a quadratic subset of $N^{-1}$ and is closed for powers and inverses. Then for each element $u_0$ in the McCrimmon radical of $N$ the identity $\mathfrak{N}-u_0 =\mathfrak{N}$ holds.
\end{lemma}

\begin{proof} Let us begin with an observation. Fix $b\in \mathfrak{N}$. Let $N_{b}$ be the homotope algebra with unit $b^{-1}\in \mathfrak{N}$. By \cite[Proposition 3]{McCrimm71} we have $$\mathcal{R}ad(N_b) = \{z\in N : U_{b} (z)\in \mathcal{R}ad(N)\},$$ and consequently $U_{b^{-1}} (u_0) = U_{b}^{-1} (u_0) \in \mathcal{R}ad(N_b)$. It follows that $$\rho_{_{N_b}}(U^{(b)}_{x} (U_{b^{-1}} (u_0)))=0, \hbox{ for all $x\in N_b$ (cf. Theorem \ref{t Aupetit c 1}).}$$ In particular 
 $\rho_{_{N_b}}(t U^{(b)}_{b^{-1}} (U_{b^{-1}} (u_0)))=0$, and thus $$ 
 t U^{(b)}_{b^{-1}} (U_{b^{-1}} (u_0)) +b^{-1}\in N_b^{-1},$$ for all $
 t\in \mathbb{R}_0^+$. Having in mind that $b^{-1}$ is the unit of the homotope $N_{b}$, we deduce that $U^{(b)}_{b^{-1}}$ is the identity mapping, and hence  $t U_{b^{-1}} (u_0) +b^{-1}\in (N_b)^{-1} = N^{-1},$ for all $t\in \mathbb{R}_0^+$. Therefore $U_{b^{-1}} (u_0) +b^{-1}$ is connected with $b^{-1}$ by a continuous path contained in $N_b^{-1}= N^{-1}$. It then follows that $U_{b^{-1}} (u_0) +b^{-1}$ is in the connected component of $N^{-1}$ containing $b^{-1}$, and thus the element $$ u_0 +b = U_b(  U_{b^{-1}} (u_0) +b^{-1}) $$ is in the connected component of $N^{-1}$ containing $b,$ and hence $u_0 +b \in \mathfrak{N}$.
 The arbitrariness of $b$ and the fact that $-b = U_{i b} (b^{-1})\in \mathfrak{N}$ give the desired conclusion.
\end{proof}

The next technical result might be known, it is included here due to the lacking of a concrete reference.

\begin{lemma}\label{l affine on segements} Let $X$ be a Banach space. Suppose $G:\mathbb{C} \to X$ is a continuous mapping satisfying:
\begin{enumerate}[$(a)$]\item $G(0) =0;$
\item The restriction of $G$ to each segment not containing zero is an affine map.
\end{enumerate} Then $G$ is real-linear.
\end{lemma}

\begin{proof} We begin by proving that \begin{equation}\label{eq G preserves antipodal} G(-\lambda) = -G(\lambda) \hbox{ for all } \lambda\in \mathbb{C}.
\end{equation} We can clearly assume that $\lambda\neq 0$. Since for $n$ big enough $0\notin [\lambda + i \frac{\lambda}{n},-\lambda + i \frac{\lambda}{n}]$, it follows from the hypothesis that $$G\left( i \frac{\lambda}{n} \right) =\frac{1}{2} G\left( \lambda + i \frac{\lambda}{n} \right) + \frac{1}{2} G\left(-\lambda + i \frac{\lambda}{n} \right).$$ Taking limit $n\to +\infty$ and applying that $G$ is continuous we derive $0 = G(0) = G(\lambda) + G(-\lambda).$\smallskip

We shall next show that \begin{equation}\label{eq Ga ffine on segement with extreme 0} \hbox{$G$ is affine (i.e., real-linear) on every segment of the form $[0,\lambda]$ with $\lambda \neq 0$.}
 \end{equation}Namely, given $0<t\leq 1$, we need to show that $G(t 0 + (1-t) \lambda) = G((1-t) \lambda) = t G(0) +(1-t) G( \lambda) =(1-t) G( \lambda)$. Pick an arbitrary $s$ with $0<s<t$. We observe that $(1-t) \lambda = r (s \lambda) + (1-r) \lambda$ for $r = \frac{t}{1-s}\in (0,1)$. By applying that $G|_{[s\lambda, \lambda]}$ is affine we get $$\begin{aligned}G((1-t) \lambda) &= G(r (s \lambda) + (1-r) \lambda) = r G(s \lambda) + (1-r) G(\lambda) \\
&= \frac{t}{1-s} G(s \lambda) + \frac{1-s-t}{1-s} G(\lambda).
\end{aligned}$$ Since $s$ was an arbitrary element in $(0,t),$ and $G$ is continuous, we can take limit $s\to 0$ in the above identity to deduce that
$ G((1-t) \lambda)=  (1-t) G(\lambda)$ as desired.\smallskip

It is now easy to check that \begin{equation}\label{eq G is homogeneous} G(\alpha \lambda) = \alpha G(\lambda) \hbox{ for all } \alpha\in \mathbb{R}, \lambda\in \mathbb{C}.
\end{equation} We can assume that $\lambda\neq 0,$ and by \eqref{eq G preserves antipodal} that $\alpha>0$. Indeed, if $0<\alpha<1$, since $G$ is affine on $[0,\lambda]$ we have $G(\alpha \lambda) = G(\alpha \lambda + (1-\alpha) 0) = \alpha G(\lambda) + (1-\alpha) G(0).$ If $\alpha>1$, by applying that $G$ is affine on $[0,\alpha \lambda]$ we get $G(\lambda)= G( \frac{1}{\alpha}\alpha \lambda +(1-\frac{1}{\alpha}) 0 ) = \frac{1}{\alpha} G( \alpha \lambda)$. The case $\alpha =1$ is clear.\smallskip

Suppose next that $0$ lies in a segment of the form $[-\alpha \lambda, \beta \lambda]$ for some $\alpha,\beta>0$ and $\lambda\in \mathbb{C}\backslash\{0\}$. Given $t\in (0,1)$ we want to prove that  $G(t (-\alpha \lambda) + (1-t) \beta \lambda) = t G(-\alpha \lambda) + (1-t) G(\beta \lambda)$. Suppose $ - t \alpha + (1-t) \beta =0$. The left-hand-side term in the desired equality is $G(0)=0$, while in the right-hand-side we have $$t G(-\alpha \lambda) + (1-t) G(\beta \lambda) =\hbox{(by \eqref{eq G is homogeneous})}= - t\alpha G( \lambda) + (1-t)\beta G( \lambda) =0.$$  Assume next that $ - t \alpha + (1-t) \beta \neq 0$. In this case we have $$\begin{aligned}G(t (-\alpha \lambda) + (1-t) \beta \lambda) &= (- t \alpha + (1-t) \beta) G(\lambda) = - t \alpha G(\lambda) + (1-t) \beta G(\lambda) \\
&= t G(-\alpha \lambda) + (1-t) G(\beta \lambda),
\end{aligned}$$ where in the first and third equality we applied \eqref{eq G is homogeneous}.\smallskip

We have proven that $G$ is real homogeneous and affine on every segment. Finally, given $\lambda,\mu\in \mathbb{C}$, it is easy to see from these properties that  $$ \frac{G(\lambda+\mu)}{2} =G\left(\frac{\lambda+\mu}{2}\right) = \frac12 (G(\lambda)+G(\mu)).$$
\end{proof}

Let us briefly equip the reader with some basic notions on numerical ranges. A (real or complex) \emph{numerical range space} is a pair $(X,u),$ where $X$ is a (real or complex) Banach space $X$ and $u$ is a fixed element in the unit sphere of $X$. The \emph{state space} of $(X,u)$ is the set $$D(X) = D(X,u) = \{\phi \in X^* : \|\phi \| = \phi (u) =1 \}.$$ The \emph{numerical range} of an element $x\in X$ is the non-empty compact and convex set defined by $$V(X,x) = V(x) =\{\phi (x) : \phi \in D(X,u)\}.$$ The \emph{numerical radius} of $x\in X$ is the number given by $$v(x) = \max\{ |\lambda| : \lambda\in V(x)\},$$ while the \emph{numerical index} of $X$ is defined by $$n(X,u)= n(X) = \inf\{ v(x) : x\in X, \ \|x\|=1 \} $$ $$= \max\{ \alpha \geq 0 : \alpha \|x\| \leq v(x) \hbox{ for all } x\in X\}.$$ The element $u$ is called a \emph{geometrically unitary} element of $X$ if and only if $n(X, u) > 0$. The celebrated Bohnenblust--Karlin theorem asserts that if $A$ is a norm-unital Banach algebra with unit $\textbf{1}$, then the numerical radius is a norm on $A$ which is equivalent to the original norm of $A$, furthermore $n(A,\textbf{1}) \geq \frac1e$ and thus $v(a) \leq \|a\|\leq e \ v(a)$ for all $a\in A$ (cf. \cite[Theorem 2.6.4]{Pal94}). It is known that the requirement concerning associativity of $A$ in the previous result can be relaxed. Namely, suppose $M$ is a norm-unital (non-necessarily associative) normed complex algebra. Then $n(M,\textbf{1})\geq \frac1e$, and thus $v(a) \leq \|a\|\leq e \ v(a)$ for all $a\in M$ (see \cite[Proposition 2.1.11]{Cabrera-Rodriguez-vol1})
.\smallskip

We can establish next a Jordan version of the result proved by O. Hatori in \cite[Theorem 3.2]{Hat2011}.

\begin{theorem}\label{t surjective isometries between invertible clopen} Let $M$ and $N$ be unital Jordan-Banach algebras. Suppose that $\mathfrak{M}\subseteq M^{-1}$ and $\mathfrak{N}\subseteq N^{-1}$ are clopen quadratic subsets of $M^{-1}$ and $N^{-1}$, respectively, which are closed for powers and inverses. Let $\Delta : \mathfrak{M}\to \mathfrak{N}$ be a surjective isometry. Then for each $w_0$ in the McCrimmon radical of $N$ the mapping $a\mapsto \Delta (a) - w_0$ is a surjective isometry from $\mathfrak{M}$ to $\mathfrak{N}.$ Furthermore, there exist a surjective real-linear isometry $T_0: M\to N$ and an element $u_0$ in the McCrimmon radical of $N$ such that $\Delta (a) = T_0(a) +u_0$ for all $a\in \mathfrak{M}$.
\end{theorem}

\begin{proof} The first statement follows from Lemma \ref{l N-u0 in N} and the fact that the translation by a fixed vector is a surjective isometry.\smallskip

Let $u_0\in \mathcal{R}ad (N)$ be the element given by Proposition \ref{p 1}. We deduce from the first conclusion that the mapping $\Delta_0 : \mathfrak{M}\to \mathfrak{N},$ $\Delta_0 (a) = \Delta (a) - u_0$ is a surjective isometry. Clearly, $\displaystyle \lim_{\mathfrak{M}\ni a\to 0} \Delta_0(a) = 0.$\smallskip

We claim that for each $a\in \mathfrak{M}$ the mapping \begin{equation}\label{eq Delta0 is real linear} \lambda \mapsto G(\lambda) = \Delta_0 (\lambda a) \hbox{ for } \lambda\in \mathbb{C}\backslash\{0\}, \hbox{ and } G(0)=0, \hbox{ is real-linear.}
\end{equation} Clearly $G$ is continuous, and if $[\lambda,\mu]$ is a segment in $\mathbb{C}$ not containing zero, by applying Lemma \ref{l Delta is locally affine} to the continuous path $\gamma: [0,1]\to \mathfrak{M}$, $\gamma(t) = t \lambda a + (1-t) \mu a$ we can deduce that $\Delta|_{[\lambda a ,\mu a]}$ is an affine mapping because $[\lambda a,\mu a]$ is convex. Lemma \ref{l affine on segements} proves the statement in \eqref{eq Delta0 is real linear}.\smallskip

Suppose $a,b\in \mathfrak{M}$ satisfy that $[a,b]\subset \mathfrak{M}$. Lemma \ref{l Delta is locally affine} can be applied to deduce that $\Delta|_{[a,b]}$ is affine. Combining this property with \eqref{eq Delta0 is real linear} we obtain \begin{equation}\label{eq Delta0 is additive on elements whose segment is in mathfrakM} \frac12 \Delta_0 (a+b) = \Delta_0\left(\frac{a+b}{2}\right) = \frac{\Delta_0(a)+\Delta_0(b)}{2},\ \forall a,b\in M \hbox{ with } [a,b]\subset \mathfrak{M}.
\end{equation}

Alike in the proof of the associative version of our result (see \cite[proof of Theorem 3.2]{Hat2011}), we consider the open convex subset $$\Omega = \bigcup_{\alpha>0} B(\alpha \textbf{1}, \alpha) = \{x\in M : \|x - \alpha \textbf{1}\|< \alpha\hbox{ for some } \alpha >0\} \subseteq M_{\mathbf{1}}^{-1}\subset \mathfrak{M}.$$ We claim that $\Delta_0 (\Omega)$ is convex. Namely, combining that $\Omega\subset\mathfrak{M}$ is convex, \eqref{eq Delta0 is real linear} and \eqref{eq Delta0 is additive on elements whose segment is in mathfrakM} we deduce that given $a,b\in \Omega$ and $t\in [0,1]$, we have $t\Delta_0(a) + (1-t) \Delta_0 (b) = \Delta(t a + (1-t) b)\in \Delta(\Omega)$, which proves the claim.\smallskip

The mapping $\Delta_0|_{\Omega} : \Omega \to \Delta_0(\Omega)$ is a surjective isometry between two open convex sets. A celebrated result by P. Mankiewicz \cite[Theorem 5 and Remark 7]{Mank1972} assures the existence of a surjective affine isometry $T_0 : M\to N$ such that $\Delta_0|_{\Omega} = T_0|_{\Omega}$.\smallskip

It is easy to see that for each $a\in \Omega$ and each $t\in \mathbb{R}^+$ the element $t a$ lies in $\Omega$ too. Therefore, $\Delta_0(t a) = T_0(t a)$. Taking limit $t\to 0$ we get $0 = T_0(0)$, which in particular assures that $T_0$ is a surjective real-linear isometry.\smallskip

We shall finally prove that $T_0|_{\mathfrak{M}} = \Delta_0$. To this end pick an arbitrary $a\in \mathfrak{M}$. Since $\|\pm a +2 \|a \| \textbf{1} - 2 \|a\| \textbf{1} \| = \|a\|< 2 \|a\|,$ we deduce that  $\pm a +2 \|a \| \textbf{1}\in \Omega$ and hence $T_{0} (\pm a +2 \|a \| \textbf{1}) = \Delta_0 (\pm a +2 \|a \| \textbf{1})$. Now, having in mind that $T_0$ is a surjective real-linear isometry and that $\Delta_0$ is an isometry, we compute \begin{equation}\label{eq final identity 1 0711} 2 \|a\| = \| a+ 2\|a\| \textbf{1} - a  \| = \|\Delta_0( a+ 2\|a\| \textbf{1}) -\Delta_0 (a) \| = \|  a+ 2\|a\| \textbf{1} -T_0^{-1} \Delta_0 (a) \|,
\end{equation} and by \eqref{eq Delta0 is real linear}
\begin{equation}\label{eq final identity 2 0711} \begin{aligned}
2 \|a\| &= \| - a+2 \|a\| \textbf{1} + a  \| = \|\Delta_0( - a+2 \|a\| \textbf{1}) -\Delta_0 (-a) \| \\ &= \| - a+2 \|a\| \textbf{1} -T_0^{-1} \Delta_0 (-a) \|= \| - a+2 \|a\| \textbf{1} +T_0^{-1} \Delta_0 (a) \|.
\end{aligned}
\end{equation} Let us take a state $\phi\in D(M,\textbf{1}) = \{\phi \in M^* : \|\phi \| = \phi (\textbf{1}) =1 \}$. By applying \eqref{eq final identity 1 0711} and \eqref{eq final identity 2 0711} it follows that $$\begin{aligned}|2 \|a\| \pm \phi (a -T_0^{-1} \Delta_0 (a) )| &= |\phi (2 \|a\| \textbf{1} \pm ( a -T_0^{-1} \Delta_0 (a)))| \\ &\leq \| 2 \|a\| \textbf{1} \pm (a - T_0^{-1} \Delta_0 (a)) \|= 2 \ \|a\|,
\end{aligned},$$ witnessing that $\phi (a -T_0^{-1} \Delta_0 (a) ) =0$. The arbitrariness of $\phi$ proves that the numerical radius of $a -T_0^{-1} \Delta_0 (a)$ is zero, that is,  $v(a -T_0^{-1} \Delta_0 (a))=0$. Finally, as we commented before this theorem, $\|a -T_0^{-1} \Delta_0 (a)\|\leq e \ v(a -T_0^{-1} \Delta_0 (a))=0$ (see \cite[Proposition 2.1.11]{Cabrera-Rodriguez-vol1}), and thus $a =T_0^{-1} \Delta_0 (a),$ which concludes the proof because $\Delta (a) = T_0 (a) + u_0$ for all $a\in \mathfrak{M}$.
\end{proof}

\begin{corollary}\label{c surjective isometries between invertible clopen} Let $M$ and $N$ be unital Jordan-Banach algebras with $N$ Jordan-semisimple. Suppose that $\mathfrak{M}\subseteq M^{-1}$ and $\mathfrak{N}\subseteq N^{-1}$ are clopen quadratic subsets of $M^{-1}$ and $N^{-1}$, respectively, which are closed for powers and inverses. Let $\Delta : \mathfrak{M}\to \mathfrak{N}$ be a surjective isometry. Then there exists a surjective real-linear isometry $T_0: M\to N$ such that $\Delta (a) = T_0(a) $ for all $a\in \mathfrak{M}$.
\end{corollary}

\begin{remark}\label{r alternative proofs with convex analysis} We would like to note that the arguments given in the proof of Theorem \ref{t surjective isometries between invertible clopen} are completely independent from those given by O. Hatori in the proof of the same result for associative unital Banach algebras in \cite[Theorem 3.2]{Hat2011}, and the prior study for unital semisimple commutative Banach algebras by the same author in \cite{Hat09}. The proof of \cite[Theorem 3.2]{Hat2011} relies on some technical results \cite[Lemmata 2.1 and 2.2]{Hat2011} based on the new proof of the Mazur--Ulam theorem given by V\"{a}is\"{a}l\"{a} \cite{Vaisala2003}. The proof here, which is also valid for associative unital Banach algebras, does not depend on the commented technical results but on an analysis of the local convex properties of $\Delta$ and an application of a celebrated result due to Mankiewicz \cite{Mank1972}. The arguments seems a bit shorter here than in the commented references.
\end{remark}

\begin{remark}\label{r difficulties in Hat11} The reader should be warned about a non-fully convincing argument in the final part of the proof of \cite[Theorem 3.2]{Hat2011}. More concretely, suppose $A$ is a  unital Banach algebra and $\mathfrak{A}$ is an open multiplicative subgroup of $A^{-1}$. As in the proof of Theorem \ref{t surjective isometries between invertible clopen} the set $$\Omega_{A} = \{x\in A : \|x - \alpha \textbf{1}\|< \alpha\hbox{ for some } \alpha >0\} \subseteq A_{\textbf{1}}^{-1}\subset \mathfrak{A}$$ is an open convex subset, and $a + 2 \|a\| \textbf{1}\in \Omega_{A}$ for all $a\in \mathfrak{A}$. However the element $- 2  i \|a\| \textbf{1}$ does not belong to $\Omega_A$. This produces subtle difficulties in the statements in \cite[pages 89, 90]{Hat2011} affirming that elements of the form $t (a + 2\|a\| \textbf{1}) + (1-t) ( \pm 2 i \|a\| \textbf{1})$ belong to $\Omega_A$ for every $0\leq t\leq 1$. So, the final paragraphs in the proof of \cite[Theorem 3.2]{Hat2011} are affected by these obstacles. The proof is clarified and simplified in the demonstration we gave above by an argument which is also valid in the associative setting. In a private communication O. Hatori indicated to us that the just commented difficulties in the proof of \cite[Theorem 3.2]{Hat2011} can be also avoided by just observing that for each $a\in A$ we have $$\left\|\frac{1}{2 (t\pm (1-t)i ) \|a\|} \left( t(a + \| a\|\textbf{1}) + (1-t) (\pm 2 i \|a\|\textbf{1})  \right) - \textbf{1}\right\|<1,$$ and hence $$\frac{1}{2 (t\pm (1-t)i ) \|a\|} \left( t(a + \| a\|\textbf{1}) + (1-t) (\pm 2 i \|a\|\textbf{1})  \right)\in \Omega_{A} \subseteq A_{\textbf{1}}^{-1}\subset \mathfrak{A}.$$ Therefore $t(a + \| a\|\textbf{1}) + (1-t) (\pm 2 i \|a\|\textbf{1})\in \mathfrak{A}$ for all $0\leq t\leq 1$.
\end{remark}

\subsection{Surjective isometries preserving the quadratic structure}\ \smallskip

We recall, once again, that the set of invertible elements in a unital Jordan-Banach algebra $M$ is not, in general, stable under Jordan products. So, contrary to what is done by O. Hatori in \cite[\S 4]{Hat2011} for surjective isometries which are also group isomorphisms between open subgroups of associative unital Banach, the problem of studding surjective isometries preserving Jordan  products between subsets of invertible elements in unital Jordan-Banach algebras does not make any sense. Instead of that, we explore surjective linear isometries preserving the quadratic expressions of the form $U_a(b)$. 

\begin{proposition}\label{p surjective isometries preserving quadratic expressions} Let $M$ and $N$ be unital Jordan-Banach algebras. Suppose that $\mathfrak{M}\subseteq M^{-1}$ and $\mathfrak{N}\subseteq N^{-1}$ are clopen quadratic subsets of $M^{-1}$ and $N^{-1}$, respectively, which are closed for powers and inverses. Let $\Delta : \mathfrak{M}\to \mathfrak{N}$ be a surjective isometry and set $u = \Delta(\textbf{1})$. We shall also assume that $\Delta$ satisfies the following property: \begin{equation}\label{property preservation of quadratic expressions} \Delta (U_a (b)) = U_{\Delta(a)} (\Delta(b)), \hbox{ for all } a,b\in \mathfrak{M}.
 \end{equation} Then there exists a real-linear isometric Jordan isomorphism $T_0: M\to N_u$ such that $\Delta (a) = T_0(a) $ for all $a\in \mathfrak{M}$, where $N_u$ stands for the $u$-homotope of $N$.
\end{proposition}

\begin{proof} 
By observing that $(N_u)^{-1} = N^{-1}$, the mapping $\Delta: \mathfrak{M}\subseteq M^{-1}\to \mathfrak{N}\subseteq (N_u)^{-1}$ can be regarded as a surjective isometry between clopen subsets of $M^{-1}$ and $(N_u)^{-1}$ satisfying the same hypotheses. To see this statement, let us take $a,b\in \mathfrak{N}$. It is known that if $a^{-1}$ denotes the inverse of $a$ in $N$, then $U_{u^{-1}} (a^{-1})$ is the inverse of $a$ in the $u$-homotope $N_u$ and lies in $\mathfrak{N}$. Furthermore, by definition
$U_{a}^{(u)} (b) =  U_{a} U_{u} (b) \in \mathfrak{N},$ because $U_{u} (b) \in \mathfrak{N}$.\smallskip

We claim that, in this case, $\Delta$ also satisfies property \eqref{property preservation of quadratic expressions} for the Jordan product in the $u$-homotope. Namely, by the assumptions,
\begin{equation}\label{eq 9 new}\begin{aligned}U^{(u)}_{\Delta(a)} (\Delta(b)) &= U_{\Delta(a)} U_{\Delta(1)} (\Delta(b)) = U_{\Delta(a)} \Delta\left( U_{1} (b)\right)\\
& = U_{\Delta(a)} \Delta\left(b\right) = \Delta\left(U_{a}(b)\right),
\end{aligned}
 \end{equation} for all $a,b\in \mathfrak{M}$, which proves the claim. Therefore $\Delta: \mathfrak{M}\subseteq M^{-1}\to \mathfrak{N}\subseteq (N_u)^{-1}$ is a unital mapping satisfying the same hypotheses. By applying Theorem \ref{t surjective isometries between invertible clopen} to the latter mapping, we deduce the existence of a surjective real-linear isometry $T_0: M\to N$ and an element $u_0$ in the McCrimmon radical of $N_u$ such that $\Delta (a) = T_0(a) + u_0$ for all $a\in \mathfrak{M}$.\smallskip

Pick $a\in \mathfrak{M}$ with $\Delta(a) = 2 u$. For each norm-null sequence $(a_n)_n$ in $\mathfrak{M}$, it follows from Proposition \ref{p 1} and \eqref{property preservation of quadratic expressions} that $$u_0 = \lim_{n} \Delta\left( U_a (a_n)\right) =\lim_{n} U^{(u)}_{\Delta(a)} (\Delta(a_n)) = \lim_{n} U^{(u)}_{2 u} (\Delta(a_n))  = \lim_{n} 2 \Delta(a_n) = 2 u_0,$$ which implies that $u_0=0$, and thus $\Delta (a) = T_0(a)$ for all $a\in \mathfrak{M}$.\smallskip

Next, we focus on the real-linear isometry $T_0: M\to N_u$. Fix an arbitrary $a\in M$. Since for $r\in\mathbb{R}$ with $|r|$ large enough we have $a+ r \textbf{1}\in M_{\mathbf{1}}^{-1}\subseteq \mathfrak{M},$ we deduce from \eqref{property preservation of quadratic expressions} and the conclusions in the previous paragraphs that \begin{equation}\label{equation 1 quadratic morphisms}\begin{aligned}U_{T_0(a+ r \textbf{1})}^{(u)} \left( T_0(a+ r \textbf{1}) \right) &= U_{\Delta(a+ r \textbf{1})}^{(u)} \left( \Delta(a+ r \textbf{1}) \right) = \Delta\left( U_{a+ r \textbf{1}} (a+ r \textbf{1}) \right)\\
&= T_0\left( U_{a+ r \textbf{1}} (a+ r \textbf{1}) \right).
\end{aligned}
\end{equation} By expanding the extreme terms in the previous identity we get
$$\begin{aligned}U_{T_0(a+ r \textbf{1})}^{(u)} \left( T_0(a+ r \textbf{1}) \right)&= U_{T_0(a) + r u }^{(u)} \left( T_0(a)+ r u \right)
= U_{T_0(a) }^{(u)} \left( T_0(a) \right) + r U_{T_0(a) }^{(u)} (u) \\
 &+ r^2  T_0(a) + r^3 u + 2 r U_{T_0(a), u }^{(u)} (T_0(a)) + 2 r^2 U_{T_0(a),u }^{(u)} (u) \\
&= U_{T_0(a) }^{(u)} \left( T_0(a) \right) + 3 r T_0(a) \circ_{u} T_0(a) + 3 r^2  T_0(a) + r^3 u
\end{aligned}  $$ and $$U_{a+ r \textbf{1}} (a+ r \textbf{1}) = U_a(a) + 3 r a^2 + 3 r^2 a + r^3 \textbf{1},$$ which combined with \eqref{equation 1 quadratic morphisms} gives $$\begin{aligned}
& U_{T_0(a) }^{(u)} \left( T_0(a) \right) + 3 r T_0(a) \circ_{u} T_0(a) + 3 r^2  T_0(a) + r^3 u \\
&= T_0(U_a(a)) + 3 r T_0(a\circ a) + 3 r^2 T_0(a) + r^3 u.
\end{aligned}$$ Now, we apply that $T_0 (b) = \Delta (b)$ for all $b\in M,$ \eqref{property preservation of quadratic expressions} and \eqref{eq 9 new} on the first terms to conclude that $T_0(a) \circ_{u} T_0(a) = T_0(a\circ a)$ for all $a\in M$. It is well known that in this case $T_0$ is an isometric Jordan isomorphism.
\end{proof}

\section{The case of JB$^*$-algebras}\label{sec: JBstar algebras}

After Hatori's studies on surjective isometries between groups of invertible elements in associative unital Banach algebras (cf. \cite{Hat09,Hat2011}), it was determined that each surjective isometry between open subgroups of the groups of all invertible elements in unital semisimple Banach algebras extends to a surjective real-linear isometry between the underlying Banach algebras. The conclusion is even better if the Banach algebras are commutative, because in such a case the real-linear extension is in fact a real isomorphism followed by multiplication by some element. In 2012, O. Hatori and K. Watanabe completed the description in the case of unital C$^*$-algebras by establishing the result that we commented at the introduction (see Theorem \ref{t HatWat}).\smallskip

JB$^*$-algebras are the Jordan alter ego of C$^*$-algebras. These structures are Jordan-Banach algebras satisfying a geometric axiom. Concretely, a \emph{JB$^*$-algebra} is a complex Jordan-Banach algebra $M$ equipped with an algebra involution $^*$ satisfying  $\|\J a{a}a \|= \|a\|^3$, $a\in M$ (where $\J a{a}a = U_a (a^*) = 2 (a\circ a^*) \circ a - a^2 \circ a^*$). A well known result in Jordan theory proves that the involution of every JB$^*$-algebra is an isometry (cf. \cite[Lemma 4]{youngson1978vidav}).\smallskip

Given a JB$^*$-algebra $M$ we shall consider the following triple product on $M$ \begin{equation}\label{eq triple product JBstar algebras} \{a,b,c\} = (a\circ b^*)\circ c + (c\circ b^*) \circ a - (a\circ c) \circ b^* \ \ (a,b,c\in M).
\end{equation} This triple product permits to see see every JB$^*$-algebra inside the wider class of JB$^*$-triple introduced in \cite{Ka83}, however we shall not make any use of this general structures.\smallskip

A JBW$^*$-algebra is a JB$^*$-algebra which is also a dual Banach space. The bidual, $M^{**},$ of every JB$^*$-algebra $M$ is a JBW$^*$-algebra with respect to a Jordan product and an involution extending the original ones in $M$ (cf. \cite[4.1.1, Theorems 4.4.3 and 4.4.16]{HOS}).\smallskip

A Jordan $^*$-homomorphism between JB$^*$-algebras $M$ and $N$ is a Jordan homomorphism $J: M\to N$ satisfying $J(a^*) = J(a)^*$ for all $a\in M$. A real-linear mapping $J: M\to N$ preserving Jordan products such that $J(a^*) = J(a)^*$ for all $a\in M$ will be called a real-linear Jordan $^*$-homomorphism. An element $p$ in a JB$^*$-algebra $M$ is called a \emph{projection} if $p= p^* = p^2$. The reader is referred to the monographs \cite{HOS} and \cite{Cabrera-Rodriguez-vol1} for the basic notions and results in the theory of JB$^*$-algebras.\smallskip

Elements $a,b$ in a C$^*$-algebra $A$ are called \emph{orthogonal} ($a\perp b$ in short) if $a b^* = b^* a =0$. This is equivalent to say that $\{a,a,b\} =0$ ($\Leftrightarrow \{b,b,a\}=0$ $\Leftrightarrow \{a,b,x\}=0$ for all $x\in A$, where the triple product is defined by $\{x,y,z\}= \frac12 (x y^* z + z y^* x)$, see for example \cite[comments in page 221]{BurFerGarMarPe08}). In the wider setting of JB$^*$-algebras, elements $a,b$ in a JB$^*$-algebra $M$ are said to be \emph{orthogonal} if $\{a,b,x\} =0$ for all $x\in M$, which is precisely the definition of ``being orthogonal'' in the JB$^*$-triple given by $M$ with the triple product in \eqref{eq triple product JBstar algebras} (cf. \cite[Lemma 1]{BurFerGarMarPe08} for additional details). It is worth to comment that if $p,q$ are projections in a JB$^*$-algebra $M$, $p\perp q$ if and only if $p\circ q=0$.\smallskip

Thought these mathematical models are widely known and studied in physics and mathematics, the definitions can be better understood by the most natural examples: every C$^*$-algebra $A$ is a JB$^*$-algebra when equipped with its natural Jordan product $a\circ b =\frac12 (a b +b a)$ and the original norm and involution. Norm-closed Jordan $^*$-subalgebras of C$^*$-algebras are called \emph{JC$^*$-algebras}. 
\smallskip

We recall for later purposes that every JB$^*$-algebra is Jordan-semisimple (cf. \cite[Lemma 4.4.28$(iii)$]{Cabrera-Rodriguez-vol1}).\smallskip

In the frame of unital JB$^*$-algebras we have an undoubted advantage with the notion of unitary. An element $u$ in a JB$^*$-algebra $M$ is called \emph{unitary} if it is invertible in the Jordan sense and its inverse coincides with $u^*$ (cf. \cite[3.2.9]{HOS} and \cite[Definition 4.1.2]{Cabrera-Rodriguez-vol1}).\smallskip

As well as in the associative context Hatori's study on surjective isometries between open subgroups of invertible elements in two Banach unit algebras \cite{Hat2011} was particularized and detailed to unital C$^*$-algebras by Hatori and Watanabe \cite{HatWan2012}, our next goal is to determine a more concrete conclusion for Theorem \ref{t surjective isometries between invertible clopen} in the case of unital JB$^*$-algebras.

\begin{theorem}\label{t surjective isometries between invertible clopen JBstar} Let $M$ and $N$ be unital JB$^*$-algebras. Suppose that $\mathfrak{M}\subseteq M^{-1}$ and $\mathfrak{N}\subseteq N^{-1}$ are clopen quadratic subsets of $M^{-1}$ and $N^{-1}$, respectively, which are closed for powers and inverses. Let $\Delta : \mathfrak{M}\to \mathfrak{N}$ be a surjective isometry. Then $\Delta(\textbf{1}) =u$ is a unitary element in $N$ and there exist a central projection $p\in M$ and a complex-linear Jordan $^*$-isomorphism $J$ from $M$ onto the $u^*$-homotope $N_{u^*}$ such that $$\Delta (a) = J(p\circ a) + J ((\textbf{1}-p) \circ a^*),$$ for all $a\in \mathfrak{M}$.\\

\noindent If we additionally assume that there exists a unitary $\omega_0$ in $N$ such that the identity $U_{\omega_0} (\Delta(\textbf{1})) = \textbf{1}$ holds, then there exist a central projection $p\in M$ and a complex-linear Jordan $^*$-isomorphism $\Phi$ from $M$ onto $N$ such that $$\Delta (a) = U_{w_0^{*}} \left(\Phi (p\circ a) + \Phi ((\textbf{1}-p) \circ a^*)\right),$$ for all $a\in \mathfrak{M}$.
\end{theorem}

\begin{proof} By Theorem \ref{t surjective isometries between invertible clopen}, or by Corollary \ref{c surjective isometries between invertible clopen}, there exists a surjective real-linear isometry $T_0: M\to N$ such that $\Delta (a) = T_0(a) $ for all $a\in \mathfrak{M}$. By \cite[Corollary 3.2]{Dang92} (or by \cite[Corollary 3.4]{FerMarPe04}) the mapping $T_0$ is a triple isomorphism, that is, $T_0$ preserves triple products of the form $\{a,b,c\} = (a\circ b^*)\circ c + (c\circ b^*) \circ a - (a\circ c) \circ b^*$ ($a,b,c\in M$). The element $u= T_0(\textbf{1})$ satisfies $\{u,u,u\} = T_0 (\{\textbf{1},\textbf{1},\textbf{1}\}) = T(\textbf{1}) =u$. Moreover, by the surjectivity of $T_0$, for each $z$ in $N$ we can find $x\in M$ with $T_0(x) =z$. Therefore, $\{u,u,z\} = T_0(\{\textbf{1},\textbf{1},x\}) = T_0(x) = z$ ($z\in N$), that is, $u$ is a unitary tripotent in $N$ in the sense employed in \cite{BraKaUp78}. We know from \cite[Proposition 4.3]{BraKaUp78} that this is equivalent to say that $u$ is a unitary in $N$.\smallskip

By \cite[Lemma 4.2.41]{Cabrera-Rodriguez-vol1} the $u^*$-homotope $N_{u^*}$ becomes a unital JB$^*$-algebra with unit $u$ for its natural Jordan product $x\circ_{u^*} y :=U_{x,y}(u^*)=\{x,u,y\}$ and the involution $*_{u}$ given by $x^{*_{u}} :=U_{u} (x^*)=\{u,x,u\}$. Since $$T_0(x\circ y ) = T_0(\{x,\textbf{1},y\}) = \{T_0(x),T_0(\textbf{1}),T_0(y)\} = T_0(x)\circ_{u^*} T_0(y)$$ and $$T_0(x^*) = T_0(\{\textbf{1},x,\textbf{1}\}) = \{T_0(\textbf{1}),T_0(x),T_0(\textbf{1})\} = T_0(x)^{*_u},\ (x,y\in M),$$ we deduce that $T_0: M\to N_{u^*}$ is an isometric real-linear unital Jordan $^*$-isomorphism.\smallskip

By a more careful reading of \cite[Corollary 3.2]{Dang92} we deduce the existence of two JB$^*$-subalgebras $M_1$ and $M_2$ of $M$ such that $M$ is the (orthogonal) direct sum $M= M_1\oplus^{\ell_{\infty}} M_2$, $T_0|_{M_1}  : M_1 \to N$ is a complex-linear triple homomorphism and $T_0|_{M_2}  : M_2 \to N$ is a conjugate-linear triple homomorphism. Let $p$ denote the unit of $M_1$. Clearly, $p$ is a central  projection in $M$ and $\textbf{1}-p$ is precisely the unit of $M_2$. The arguments in the previous paragraphs show that $T_0|_{M_1}  : M_1 \to N_u$ is a complex-linear Jordan $^*$-monomorphism and $T_0|_{M_2}  : M_2 \to N_u$ is a conjugate-linear Jordan $^*$-monomorphism. Furthermore, bearing in mind that $M_1$ and $M_2$ are orthogonal in $M$ and that $T_0: M\to N_u$ is a real-linear Jordan $^*$-isomorphism, we deduce that $N_1 = T_0(M_1)$ and $N_2 = T_0(M_2)$ are orthogonal JB$^*$-subalgebras of $N_u$ with $N_u = N_1\oplus^{\ell_{\infty}} N_2$. Moreover, the mapping $J: M\to N_u$, $$J(x) = T_0(p\circ x) + T_0((\textbf{1}-p) \circ x)^{*_u} = T_0(p\circ x) + T_0((\textbf{1}-p) \circ x^*)\ \ (x\in M)$$ is a complex-linear Jordan $^*$-isomorphism (just apply that $x\mapsto p\circ x$ and $x\mapsto (\textbf{1}-p) \circ x$ are the natural projections of $M$ onto $M_1$ and $M_2$, respectively), and the identity $$\Delta (a) = J(p\circ a) + J ((\textbf{1}-p) \circ a^*),$$ holds for all $a\in \mathfrak{M}$.\smallskip

Suppose, finally, that we can find a unitary $w_0\in N$ such that $U_{w_0} (\Delta(\textbf{1}))= \textbf{1}$. The mapping $U_{w_0} : N\to N$ is a surjective complex-linear isometry and a triple isomorphism mapping $u$ to $\textbf{1}$ (cf. \cite[Theorem 4.2.28]{Cabrera-Rodriguez-vol1}), therefore $U_{w_0} : N_u\to N$ is a complex-linear Jordan $^*$-isomorphism. Let $p\in M$ and $J: M\to N_u$ be the central projection and the complex-linear Jordan $^*$-isomorphism given by our first conclusions. Clearly, $\Phi= U_{w_0} \circ J: M\to N$ is a Jordan $^*$-isomorphism and the equality $$\Delta (a) = U_{w_0^{-1}} \left(\Phi (p\circ a) + \Phi ((\textbf{1}-p) \circ a^*)\right)= U_{w_0^{*}} \left(\Phi (p\circ a) + \Phi ((\textbf{1}-p) \circ a^*)\right),$$ holds for all $a\in \mathfrak{M}$.
\end{proof}

\begin{remark}\label{r alternative proof JBstar} It should be remarked that Theorem \ref{t surjective isometries between invertible clopen JBstar} can be also deduced, via an alternative argument, from the conclusions in the recent papers \cite{CuPe20,CuPe20a}. Let $\Delta: \mathfrak{M}\to \mathfrak{N}$ be a surjective isometry in the conditions of the just quoted theorem. Let $T_0: M\to N$ be the surjective real-linear isometry given by Theorem \ref{t surjective isometries between invertible clopen} or by Corollary \ref{c surjective isometries between invertible clopen}. Let $\partial_e(\mathcal{B}_M)$ and $\partial_e(\mathcal{B}_N)$ denote the sets of all extreme points of the closed unit ball of $M$ and $N$, respectively. Clearly $T_0$ maps $\partial_e(\mathcal{B}_M)$ onto $\partial_e(\mathcal{B}_N)$.\smallskip

Let $w$ be an extreme point of the closed unit ball of a unital JB$^*$-algebra $M^\prime$. Theorem 3.8 in \cite{CuPe20} proves that $w$ is a unitary if and only if the set $$\mathcal{M}^\prime_{w} = \{e\in \partial_e(\mathcal{B}_{M^\prime}) : \|w\pm e\|\leq \sqrt{2} \}$$ contains an isolated point. Let $\mathcal{U}(M)$ and $\mathcal{U}(N)$ denote the sets of unitary elements in $M$ and $N$, respectively. The just quoted result and the observations in the previous paragraph show that $T_0(\mathcal{U}(M)) = \mathcal{U}(N)$, and hence $T_0|_{\mathcal{U}(M)} : \mathcal{U}(M)\to \mathcal{U}(N)$ is a surjective isometry. In particular, $u= T_0(\textbf{1})$ is a unitary in $N$. Let $\mathcal{U}(N_{u^*})$ stand for the set of unitaries in the $u^*$-homotope of $N$. It is known that $\mathcal{U}(N_{u^*}) =\mathcal{U}(N)$ {\rm(}cf. \cite[Lemma 4.2.41$(ii)$]{Cabrera-Rodriguez-vol1} or \cite[Proposition 4.3]{BraKaUp78}{\rm)}. We therefore conclude that $T_0|_{\mathcal{U}(M)} : \mathcal{U}(M)\to \mathcal{U}(N_{u^*})$ is a surjective isometry. By \cite[Corollary 3.5]{CuPe20a} there exist a central projection $p$ in $M$ and an isometric Jordan $^*$-isomorphism $J:M\to N_{u^*}$ such that \begin{equation}\label{eq T0 and the real Jordan isom coincide on the principal component} T_0( w ) =  J(p \circ w) + J( (\textbf{1}-p)\circ w^*),
\end{equation} for all $w\in \exp(i M_{sa})\subseteq M_{\mathbf{1}}^{-1}\subseteq \mathfrak{M}$.\smallskip

On the other hand, the Russo-Dye type theorem for unital JB$^*$-algebras proved by J.D.M. Wright and M. Youngson in \cite[Corollary 2.4]{WriYou77} {\rm(}see also \cite[Corollary 3.4.7]{Cabrera-Rodriguez-vol1}{\rm)} asserts that the closed unit ball of $M$ coincides with the closed convex-hull of the set $\exp(i M_{sa})$. Since the maps $T_0$ and $x\mapsto J(p \circ x + (\textbf{1}-p)\circ x^*)$ are real-linear and continuous, we deduce from \eqref{eq T0 and the real Jordan isom coincide on the principal component} that $T_0( x ) =  J(p \circ x) + J( (\textbf{1}-p)\circ x^*),$ for all $x\in M$, and consequently, $\Delta( a ) =  J(p \circ a) + J( (\textbf{1}-p)\circ a^*),$ for all $a\in \mathfrak{M}$. The rest follows from the arguments given in the final part of the proof of Theorem \ref{t surjective isometries between invertible clopen JBstar}.
\end{remark}

\begin{remark} In the same way that in Remark \ref{r alternative proof JBstar} we have given an alternative proof of Theorem \ref{t surjective isometries between invertible clopen JBstar} from \cite[Theorem 3.8]{CuPe20}, \cite[Corollary 3.5]{CuPe20a} and an appropriate version of the Russo-Dye theorem for unital JB$^*$-algebras, the C$^*$- version of our Theorem \ref{t surjective isometries between invertible clopen JBstar} in \cite[Theorem 2.2]{HatWan2012} can be also derived from \cite[Lemma 3.1]{Mori2018} and \cite[Theorem 1]{HatMol2014}.
\end{remark}

\medskip\medskip

\textbf{Acknowledgements} Author partially supported by the Spanish Ministry of Science, Innovation and Universities (MICINN) and European Regional Development Fund project no. PGC2018-093332-B-I00, Programa Operativo FEDER 2014-2020 and Consejer{\'i}a de Econom{\'i}a y Conocimiento de la Junta de Andaluc{\'i}a grant number A-FQM-242-UGR18, and Junta de Andaluc\'{\i}a grant FQM375.\smallskip

The author would like to express his gratitude to the anonymous referees who generously spent their precious time in the review process of this paper. Referees' contributions are fundamental in the world of academia, in this case their valuable comments and suggestions allowed us to elude certain difficulties in some of the original arguments, and improved the final form of the manuscript.


\begin{thebibliography}{0}
\bibitem{Aupetit93} B. Aupetit, Spectral characterization of the radical in Banach and Jordan-Banach algebras, \emph{Math. Proc. Cambridge Philos. Soc.} \textbf{114}, no. 1, 31--35 (1993).

\bibitem{Aupetit95} B. Aupetit, Spectral characterization of the sode in Jordan-Banach algebras, \emph{Math. Proc. Camb. Phil. Soc.} \textbf{117}, 479-489 (1995).

\bibitem{BonDunCNA73} F.F. Bonsall, J. Duncan, \emph{Complete normed algebras}. Ergebnisse der Mathematik und ihrer Grenzgebiete, Band 80. Springer-Verlag, New York-Heidelberg, 1973.

\bibitem{BraKaUp78} R. Braun, W. Kaup, H. Upmeier, A holomorphic characterisation of Jordan-C$^*$-algebras, \emph{Math. Z.} \textbf{161}, 277--290 (1978).

\bibitem{BurFerGarMarPe08} M. Burgos, F.J. Fern{\' a}ndez-Polo, J.J Garc{\'e}s, J. Mart{\'i}nez Moreno, A.M. Peralta, Orthogonality preservers in C$^*$-algebras, JB$^*$-algebras and JB$^*$-triples, \emph{J. Math. Anal. Appl.} \textbf{348}(1), 220--233 (2008).

\bibitem{Cabrera-Rodriguez-vol1} {M. Cabrera~Garc\'{\i}a, A. Rodr\'{\i}guez~Palacios,}
\newblock {\em Non-associative normed algebras. {V}ol. 1}, vol.~154 of {\em
  Encyclopedia of Mathematics and its Applications}.
\newblock Cambridge University Press, Cambridge, 2014.
\newblock The Vidav-Palmer and Gelfand-Naimark theorems.

\bibitem{CuPe20} M. Cueto-Avellaneda, A.M. Peralta, Metric characterisation of unitaries in JB$^*$-algebras, \emph{Mediterr. J. Math.} \textbf{17}:124 (2020). DOI: 10.1007/s00009-020-01556-w

\bibitem{CuPe20a} M. Cueto-Avellaneda, A.M. Peralta, Can one identify two unital JB$^*$-algebras by the metric spaces determined by their sets of unitaries?, preprint 2020. arXiv:2005.04794

\bibitem{Dang92} T. Dang, Real isometries between JB$^*$-triples, \emph{Proc. Amer. Math. Soc.} \textbf{114}, no. 4, 971--980 (1992).

\bibitem{FerMarPe04} F.J. Fern{\'a}ndez-Polo, J. Mart{\' i}nez, A.M. Peralta, Surjective isometries between real JB$^*$-triples, \emph{Math. Proc. Cambridge Phil. Soc.}, \textbf{137} 709--723 (2004).

\bibitem{HOS} H. Hanche-Olsen, E. St{\o}rmer, \emph{Jordan Operator Algebras}, Pitman, London, 1984.

\bibitem{Hat09} O. Hatori, Isometries between groups of invertible elements in Banach algebras, \emph{Studia Math.} \textbf{194}, 293--304 (2009).

\bibitem{Hat2011} O. Hatori, Algebraic properties of isometries between groups of invertible elements in Banach algebras, \emph{J. Math. Anal. Appl.} \textbf{376}, no. 1, 84--93 (2011).

\bibitem{HatMol2012} O. Hatori, L. Moln\'{a}r, Isometries of the unitary group, \emph{Proc. Amer. Math. Soc.} \textbf{140}, no. 6, 2127--2140 (2012).

\bibitem{HatMol2014} O. Hatori, L. Moln\'{a}r, Isometries of the unitary groups and Thompson isometries of the spaces of invertible positive elements in $C^*$-algebras, \emph{J. Math.\ Anal.\ Appl.} \textbf{409}, 158--167 (2014).


\bibitem{HatWan2012} O. Hatori, K. Watanabe, Isometries between groups of invertible elements in C$^*$-algebras, \emph{Studia Math.} \textbf{209}, 103--106 (2012).

\bibitem{HogbMcCrimm81} L. Hogben, K. McCrimmon, Maximal modular inner ideals and the Jacobson radical of a Jordan algebra, \emph{J. Algebra} \textbf{68}, no. 1, 155--169 (1981).


\bibitem{Jacobson81} N. Jacobson, \emph{Structure theory of Jordan algebras}, Lecture Notes in Math., vol. 5, The University of Arkansas, 1981.

\bibitem{Ka83} W. Kaup, A Riemann mapping theorem for bounded symmentric domains in complex Banach spaces, \emph{Math. Z.} \textbf{183}, 503--529 (1983).

\bibitem{Loos96} O. Loos, On the set of invertible elements in Banach Jordan algebras, \emph{Results Math.} \textbf{29}, no. 1-2, 111--114 (1996).

\bibitem{Mank1972} P. Mankiewicz, On extension of isometries in normed linear spaces, \emph{Bull. Acad. Pol. Sci., S\'{e}r. Sci. Math. Astron. Phys.} \textbf{20}, 367--371 (1972).


\bibitem{McCrimm66} K. McCrimmon, A general theory of Jordan rings, \emph{Proc. Nat. Acad. Sci. U.S.A.} \textbf{56}, 1072--1079 (1966).

\bibitem{McCrimm69} K. McCrimmon, The radical of a Jordan algebra, \emph{Proc. Nat. Acad. Sci. USA.} \textbf{62}, 67l--678 (1969).

\bibitem{McCrimm71} K. McCrimmon, A characterization of the radical of a Jordan algebra, \emph{J. Algebra} \textbf{18}, 103--111 (1971).

\bibitem{Mori2018} M. Mori, Tingley's problem through the facial structure of operator algebras, \emph{J. Math. Anal. Appl.} \textbf{466}, no. 2, 1281--1298 (2018).

\bibitem{Niem77} T. Nieminen, On ultrapseudocompact and related spaces, \emph{Ann. Acad. Sci. Fenn. Ser. A I Math.} \textbf{3}, no. 2, 185--205 (1977).

\bibitem{Pal94} T.W. Palmer, \emph{Banach Algebras and the General Theory of $^*$-Algebras. Vol. I. Algebras and Banach Algebras}, Encyclopedia Math. Appl., vol. 49, Cambridge University Press, Cambridge, 1994.

\bibitem{Vaisala2003} J. V\"{a}is\"{a}l\"{a}, A proof of the Mazur-Ulam theorem, \emph{Amer. Math. Monthly} \textbf{110} (7), 633--635 (2003).


\bibitem{WriYou77} J.D.M. Wright, M.A. Youngson, A Russo--Dye theorem for Jordan C$^*$-algebras. In \emph{Functional analysis: surveys and recent results (Proc. Conf., Paderborn, 1976)}, pp. 279-282. North-Holland Math. Studies, Vol. 27; Notas de Mat., No. 63, North-Holland, Amsterdam, 1977.

\bibitem{youngson1978vidav} M.A. Youngson, A {V}idav theorem for {B}anach {J}ordan algebras, \emph{Math. Proc. Cambridge Philos. Soc.} \textbf{84}, 2, 263--272 (1978).

\bibitem{Zelaz73} W. \.{Z}elazko, \emph{Banach Algebras}, Elsevier, 1973.
\end{thebibliography}
\end{document}